\documentclass[a4paper,reqno]{amsart}

\usepackage{amsmath}
\usepackage{amsfonts}
\usepackage{amssymb}
\usepackage{amsthm}

\numberwithin{equation}{section}

\theoremstyle{plain}
\newtheorem{theorem}{Theorem}[section]

\newtheorem{lemma}[theorem]{Lemma}
\newtheorem{proposition}[theorem]{Proposition}
\newtheorem{corollary}[theorem]{Corollary}
\theoremstyle{definition}
\newtheorem{definition}[theorem]{Definition}
\newtheorem{assumption}[theorem]{Assumption}
\newtheorem{example}[theorem]{Example}
\theoremstyle{remark}
\newtheorem{remark}[theorem]{Remark}
\newcounter{counter_a}
\newenvironment{myenum}{\begin{list}{{\rm(\roman{counter_a})}}%
{\usecounter{counter_a}
\setlength{\itemsep}{0.5ex}\setlength{\topsep}{0.7ex}
\setlength{\leftmargin}{5ex}\setlength{\labelwidth}{5ex}}}{\end{list}}

\DeclareMathOperator\Real{Re}
\DeclareMathOperator\Imag{Im}
\renewcommand\Re{\Real}
\renewcommand\Im{\Imag}

\newcommand\cH{\mathcal H}
\newcommand\cK{\mathcal K}

\newcommand\cL{\mathcal L}
\newcommand\cM{\mathcal M}

\newcommand\cO{\mathcal O}

\newcommand\rmo{{\rm o}}
\newcommand\BB{\mathbb B}
\newcommand\CC{\mathbb C}

\newcommand\NN{\mathbb N}
\newcommand\RR{\mathbb R}
\newcommand\UU{\mathbb U}

\newcommand\fra{\mathfrak a}
\newcommand\frd{\mathfrak d}
\newcommand\frs{\mathfrak s}
\newcommand\frt{\mathfrak t}

\newcommand\eps{\varepsilon}

\newcommand\ov{\overline}

\newcommand{\defeq}{\mathrel{\mathop:}=}
\newcommand{\eqdef}{=\mathrel{\mathop:}}
\newcommand{\defequ}{\mathrel{\mathop:}\hspace*{-0.72ex}&=}

\newcommand\sigmaess{\sigma_{\rm ess}}
\newcommand\sigmap{\sigma_{\rm p}}
\newcommand\sigmaapp{\sigma_{\rm app}}
\newcommand\sigmadis{\sigma_{\rm dis}}

\newcommand{\dist}{\mathrm{dist}}
\newcommand\lae{\lambda_{\rm e}}
\newcommand\VM{{\rm\textbf{(VM{\boldmath$^-$}\!)}}}

\newcommand\condA{{\rm\textbf{(A)}}}
\newcommand\BBnr{\BB_{\rm nr}}

\DeclareMathOperator\dom{dom}
\DeclareMathOperator\ran{ran}

\DeclareMathOperator\defect{def}
\DeclareMathOperator\nul{nul}
\DeclareMathOperator\ind{ind}
\DeclareMathOperator\essran{ess\,ran}
\DeclareMathOperator\essinf{ess\,inf}
\DeclareMathOperator\esssup{ess\,sup}

\newcommand\void[1]{}

\begin{document}

\title[Spectral properties of unbounded $J$-self-adjoint matrices]{Spectral properties of
unbounded \\[0.5ex]
{\boldmath$J$}-self-adjoint block operator matrices}

\author{Matthias Langer}

\address{Department of Mathematics and Statistics,
University of Strathclyde \newline
26 Richmond Street, Glasgow G1 1XH, United Kingdom \newline
\texttt{m.langer@strath.ac.uk} \newline
(corresponding author)}

\author{Michael Strauss}

\address{Department of Mathematics, University of Sussex \newline
Falmer Campus, Brighton BN1 9QH, United Kingdom \newline
\texttt{m.strauss@sussex.ac.uk}}

\begin{abstract}
We study the spectrum of unbounded $J$-self-adjoint block operator matrices.
In particular, we prove enclosures for the spectrum, provide a sufficient
condition for the spectrum being real and derive variational principles
for certain real eigenvalues even in the presence of non-real spectrum.
The latter lead to lower and upper bounds and asymptotic estimates
for eigenvalues.
\\[1ex]
\\[1ex]
\end{abstract}


\maketitle

\section{Introduction}

\noindent
Let $\cH_1$ and $\cH_2$ be Hilbert spaces and consider a block operator matrix
acting in the direct sum $\cH\defeq\cH_1\oplus\cH_2$, i.e.\ an operator of the form
\[
  \cM_0 = \begin{pmatrix} A & B \\ C & D \end{pmatrix},
\]
where, e.g.\ $A$ is an operator in $\cH_1$ and $B$ an operator from $\cH_2$ to $\cH_1$.
Such operators play an important role in many spectral problems and their
applications; see, e.g.\ the monograph \cite{tretbook} and the references cited therein.
In recent years, many papers have studied and described spectral properties of such
block operator matrices in terms of their entries $A$, $B$, $C$ and $D$.
In particular, spectral enclosures and variational principles for
characterising eigenvalues, often in a gap in the essential spectrum,
have received a great deal of attention;
see, e.g.\ \cite{AL95,AMS09,AMT10,ALMS94,DES00,GLS99,JT02,KLT04,LLMT05,LLT02,
LMMT01,LMM90,LT98,MS96,T09}.
In many of these papers the case was studied
when $A$ and $D$ are self-adjoint and $C=B^*$, in which case $\cM_0$ is a
symmetric operator in $\cH$, and often even essentially self-adjoint.

In the present paper we consider the situation when $A$ and $D$ are self-adjoint
and $C=-B^*$.  In this case the operator $\cM_0$ is $J$-symmetric where
$J=\bigl(\begin{smallmatrix} I & 0 \\ 0 & -I \end{smallmatrix}\bigr)$;
this means that $J\cM_0$ is a symmetric operator in $\cH$, or in other words,
the operator $\cM_0$ is symmetric in the Krein space $\cK\defeq\cH_1\oplus\cH_2$
with indefinite inner product
$[x,y]\defeq\langle Jx,y\rangle$, where $\langle \cdot\,,\cdot\rangle$
denotes the inner product in the Hilbert space $\cH$.
Every bounded self-adjoint operator in a Krein space can be written
as a block operator matrix with $A$, $D$ self-adjoint and $C=-B^*$.
However, this is not true in general for unbounded operators.
Moreover, for given self-adjoint $A$, $D$ and $C=-B^*$ it is not guaranteed
that $\cM_0$ has a closure that is self-adjoint in the Krein space.
Even if the latter is true, it is not clear whether this closure
has non-empty resolvent set.

We consider two classes of unbounded block operator matrices:
certain upper dominant matrices (where the operators in the top row,
i.e.\ $A$ and $B$ are stronger than those in the bottom row
in the sense that the latter are relatively bounded with respect to the former)
and certain diagonally dominant matrices (where the stronger operators
are the diagonal operators $A$ and $D$).
In these situations the operator $\cM_0$ is closable, its closure $\cM$
is $J$-self-adjoint, i.e.\ self-adjoint in the Krein space,
and it has non-empty resolvent set.
Certain diagonally dominant $J$-self-adjoint block operator matrices,
often with bounded $B$ or some other extra assumptions, have been investigated,
e.g.\ in \cite{AL95,AMS09,AMT10,JT02,LLMT05,LLMT08,LLT02,T09}.
However, to our knowledge, upper dominant $J$-self-adjoint block operator matrices
have been studied only in few papers; see \cite{JT02,MS96}.

Since in both cases that we consider (upper and diagonally dominant case)
the operator $A$ is stronger in some sense than $C=-B^*$, there exist $a\in\RR$
and $b\ge0$ such that $\|B^*x\|^2 \le a\|x\|^2+b\langle Ax,x\rangle$ for all
$x\in\dom(A)$.
Using these constants $a,b$ and the location of the spectra of $A$ and $D$
we prove enclosures for the spectrum of $\cM$.
In particular, the non-real spectrum is always contained in a compact set
and hence the resolvent set is non-empty; see Theorem~\ref{th:specinB}.
We also give a sufficient condition for the spectrum of $\cM$ being real,
namely condition \condA{} introduced in Definition~\ref{def:Betc}.
In the latter situation we can give an enclosure that consists
of one interval (in a limiting case) or of two disjoint intervals (in the generic case).
The main tool for proving these enclosures is the
quadratic numerical range $W^2(\cM)\subset\CC$, which was introduced in \cite{LT98}
and whose closure contains the spectrum in many situations;
see Definition~\ref{def:W2} and Proposition~\ref{pr:spec_incl}.

The second set of results concerns the characterisation of certain eigenvalues
with variational principles.  Instead of the classical Rayleigh quotient
we use either a functional that is connected with the quadratic numerical range
(see Theorem~\ref{finitevar2}) or a generalised Rayleigh functional that is
associated with the Schur complement of the block operator matrix
(see Theorem~\ref{finitevar}); the Schur complement is formally given by
\[
  S(z) = A-z+B(D-z)^{-1}B^*
\]
and is an operator function acting only in the first component $\cH_1$.
With the help of these variational principles we also prove enclosures
for eigenvalues of $\cM$ as well as asymptotic enclosures under the extra
assumption that $A$ has compact resolvent.

Further, we prove some results concerning the properties of $\cM$ considered
as an operator in a Krein space.  In particular, we prove that spectral
points in a certain interval are of positive type, and therefore there exists
a local spectral function for the operator~$\cM$.
If $A$ has compact resolvent, then $\cM$ is definitisable.
Finally, we discuss some examples with differential operators as entries
to illustrate our results.

Let us give a brief synopsis of the paper.
In Section~\ref{jsa_sec} we define the operator $\cM$, which is the
closure of the block operator matrix $\cM_0$, describe its domain and action
and show that it is $J$-self-adjoint.
The Schur complement $S$ of $\cM$ is introduced and studied in Section~\ref{schur}.
In particular, in Theorem~\ref{specequiv} we show that the spectra of $\cM$ and $S$
coincide on the set where $S$ is defined.
In Section~\ref{numrange_sec} the quadratic numerical range $W^2(\cM)$ of $\cM$
is introduced and used to show that the spectrum of $\cM$ is contained
in the set $\BB$ that is defined in Definition~\ref{def:Betc};
see Theorem~\ref{th:specinB}.
A number $\mu\in\RR$ also plays an important role in the definition of $\BB$
(real parts of non-real points in $\BB$ are bounded from above by $\mu$)
and in later sections.
Section~\ref{varprince} is devoted to the characterisation of eigenvalues
in $(\mu,\infty)$ via variational principles: Theorems~\ref{finitevar}
and \ref{finitevar2} use functionals that are connected with the Schur complement
and the quadratic numerical range, respectively.
These characterisations are used in Section~\ref{estimates} to obtain
enclosures for eigenvalues in the interval $(\mu,\infty)$.
In Section~\ref{krein} we prove that spectral points in $(\mu,\infty)$
are of positive type, we show that if a strict version of condition \condA{}
is satisfied, then $\cM-\gamma$ is non-negative in the Krein space for certain $\gamma$,
and we prove that $\cM$ is definitisable if $A$ has compact resolvent.
Finally, in Section~\ref{exs} we apply our results to some examples where the
entries of the block operator matrix are differential and multiplication operators.

\textbf{Notation.}
For a linear operator $T$ we denote its spectrum by $\sigma(T)$ and
its resolvent set by $\rho(T)$.
In addition, we define the \emph{essential spectrum}, \emph{point spectrum},
\emph{discrete spectrum}, \emph{approximate point spectrum} and the
\emph{numerical range} as follows:
\begin{align*}
  \sigmaess(T) &\defeq \big\{z\in\mathbb{C}: T-z\text{ is not Fredholm}\big\}, \\[1ex]
  \sigmap(T) &\defeq \big\{z\in\mathbb{C}: \ker(T-z)\ne\{0\}\big\}, \displaybreak[0]\\[1ex]
  \sigmadis(T) &\defeq \big\{z\in\sigmap(T): T-z \text{ is Fredholm and $z$ is isolated in }\sigma(T)\big\}, \displaybreak[0]\\[1ex]
  \sigmaapp(T) &\defeq \bigl\{z\in\CC: \exists\,x_n\in\dom(T),\,\|x_n\|=1,\,(T-z)x_n\to0\bigr\}, \\[1ex]
  W(T) &\defeq \bigl\{\langle Tx,x\rangle: x\in\dom(T),\,\|x\|=1\bigr\}.
\end{align*}
The square root of a real number is defined such that $\sqrt{t}\ge0$ for $t\in[0,\infty)$
and $\Im\sqrt{t}>0$ for $t\in(-\infty,0)$.
Moreover, we use the notation $(t)_+ \defeq \max\{t,0\}$ for $t\in\RR$.

\section{$J$-self-adjoint operator matrices}\label{jsa_sec}

\noindent
Throughout this paper let $\mathcal{H}_1$ and $\mathcal{H}_2$ be Hilbert spaces
with inner products $\langle\cdot\,,\cdot\rangle$;
we also denote the inner product in $\cH\defeq\cH_1\oplus\cH_2$ by $\langle\cdot\,,\cdot\rangle$.
Moreover, let $A$ be a self-adjoint operator acting in $\mathcal{H}_1$ which is bounded from below;
let $B$ be a densely defined and closable operator acting from $\mathcal{H}_2$ to $\mathcal{H}_1$;
and let $D$ be a self-adjoint operator acting in $\mathcal{H}_2$ which is bounded from above.
Let $\fra$ and $\frd$ be the closed quadratic forms associated with the operators $A$ and $D$,
respectively, and set
\begin{equation}\label{defalmidepl}
  \alpha_- \defeq \min\sigma(A), \qquad \delta_+ \defeq \max\sigma(D).
\end{equation}

We shall be concerned with the spectral properties of (the closure of) the
block operator matrix
\begin{equation}\label{M0}
  \cM_0 \defeq \begin{pmatrix} A & B \\ -B^* & D \end{pmatrix}:
  \mathcal{H}_1\oplus\mathcal{H}_2\to\mathcal{H}_1\oplus\mathcal{H}_2
\end{equation}
with $\dom(\cM_0) = \bigl(\dom(A)\cap\dom(B^*)\bigr)\times\bigl(\dom(B)\cap\dom(D)\bigr)$.
We consider two classes of matrices, which are introduced in the following assumption.

\begin{assumption}\label{assump}
Let $A$, $B$, $D$ and $\cM_0$ be as above.
We assume that one of the following conditions is satisfied:
\begin{enumerate}
\setlength{\itemsep}{2ex}
\item[(I)]
$\dom(\fra)\subset\dom(B^*), \qquad \dom(B)\subset\dom(D)$,
\qquad $\dom(B)$ is a core for $D$;
\item[(II)]
$\dom(\fra)\subset\dom(B^*), \qquad \dom(\frd)\subset\dom(B)$.
\end{enumerate}
\end{assumption}

\medskip

\noindent
Under Assumption~\ref{assump}.(I) the block operator matrix $\cM_0$
is upper dominant in the sense that the operators in the second row
are relatively bounded with respect to the operators in the first row;
see \cite[Definition~2.2.1]{tretbook}.
If Assumption~\ref{assump}.(II) is satisfied, then $\cM$
is diagonally dominant.
As we shall see below, $\cM_0$ is closed in case (II) and
closable in case (I).  In both cases, we denote the closure of $\cM_0$ by $\cM$.

The condition $\dom(\fra)\subset\dom(B^*)$
(which is satisfied in both cases (I) and (II))
ensures the existence of constants $a\in\RR$ and $b\ge0$ such that
\begin{equation}\label{domcons}
  \|B^*x\|^2\le a\|x\|^2 + b\fra[x] \qquad\text{for all}\;\; x\in\dom(\fra).
\end{equation}
Clearly, one can choose $a$ and $b$ such that both are non-negative, but we allow $a$
to be negative to have more flexibility in our estimates.
Moreover, let $b_0$ be the relative bound, i.e.\
\begin{equation}\label{defb0}
  b_0 \defeq \inf\bigl\{b\ge0:\; \text{there exists an $a\in\RR$
  such that \eqref{domcons} holds}\bigr\}.
\end{equation}
However, for many theorems, in particular, in later sections, we fix one
pair $a,b$ such that \eqref{domcons} holds.

\begin{remark}\label{re:abalmi}
Relation \eqref{domcons} implies that, for $x\in\dom(\fra)\backslash\{0\}$,
\[
  0 \le \frac{\|B^*x\|^2}{\|x\|^2} \le a + b\frac{\fra[x]}{\|x\|^2}\,.
\]
Taking the infimum of the right-hand side over all $x\in\dom(\fra)\backslash\{0\}$
we obtain
\begin{equation}\label{abalmi}
  a+b\alpha_- \ge 0.
\end{equation}
\end{remark}

In the following we shall often use the boundedness of certain operators.
Let $\nu<\min\sigma(A)$.  The condition $\dom(\fra)\subset\dom(B^*)$ and the
closed graph theorem imply that $B^*(A-\nu)^{-\frac{1}{2}}$ is bounded and
everywhere defined. Hence $(A-\nu)^{-\frac{1}{2}}B$
is bounded and densely defined and
\begin{align}
  & \bigl((A-\nu)^{-\frac{1}{2}}B\bigr)^*=B^*(A-\nu)^{-\frac{1}{2}},
  \label{rel_ABstar1} \\[1ex]
  & \ov{(A-\nu)^{-1}B} = (A-\nu)^{-\frac{1}{2}}\ov{(A-\nu)^{-\frac{1}{2}}B}
  \label{rel_ABstar2}
\end{align}
hold.

\begin{remark}\label{re:comp1}
If the operator $B^*(A-\nu)^{-\frac12}$ is compact for some $\nu<\min\sigma(A)$,
then $B^*$ is $(A-\nu)^{\frac12}$-bounded with relative bound $0$;
see, e.g.\ \cite[Corollary~III.7.7]{EdmundsEvans}.
This implies that \eqref{domcons} holds for arbitrary $b>0$
(see, e.g.\ \cite[\S V.4.1]{katopert})
and hence $b_0=0$ in this case.
\end{remark}

In the next theorem we explicitly describe the domain and the action of the
closure $\cM$ of $\cM_0$.  In the proof we reduce the problem to a situation
with a self-adjoint operator in a Hilbert space.  To this end, define the matrix
\begin{equation}\label{defJ}
  J \defeq \begin{pmatrix} I & 0 \\ 0 & -I \end{pmatrix}:
  \cH_1\oplus\cH_2\to\cH_1\oplus\cH_2.
\end{equation}

\begin{theorem}\label{jsaness}
If Assumption~{\rm\ref{assump}.(I)} is satisfied, then $J\cM_0$ is essentially self-adjoint
and $\cM_0$ is closable with closure $\cM$.

If Assumption~{\rm\ref{assump}.(II)} is satisfied, then $J\cM_0$ is self-adjoint
and $\cM_0$ is closed with domain $\dom(A)\times\dom(D)$ and hence equal to $\cM$.

Let $\nu<\min\sigma(A)$ be arbitrary.
In both cases {\rm(I)} and {\rm(II)} we have
\begin{align}
\label{top1}
  \dom(\cM) &= \bigg\{\binom{x}{y}: y\in\dom(D),\;x + \overline{(A-\nu)^{-1}B}y\in\dom(A)\bigg\},
  \\[1ex]
  \label{top2}
  \cM\binom{x}{y} &= \begin{pmatrix}
    (A-\nu)\bigl(x+\overline{(A-\nu)^{-1}B}y\bigr) + \nu x \\[1ex]
    -B^*x + Dy
  \end{pmatrix},
  \qquad \binom{x}{y}\in\dom(\cM).
\end{align}
If\, $(x,y)^T\in\dom(\cM)$, then $x\in\dom(\fra)$.
Moreover, for $(x,y)^T\in\dom(\cM)$
and $(\hat x,\hat y)^T\in\dom(\fra)\times\cH_2$ we have
\begin{equation}\label{connection_T_form}
  \biggl\langle \cM\binom{x}{y},\binom{\hat x}{\hat y}\biggr\rangle
  = \fra[x,\hat x] + \langle y,B^*\hat x\rangle
  - \langle B^*x,\hat y\rangle + \langle Dy,\hat y\rangle.
\end{equation}
\end{theorem}

\begin{proof}
For the self-adjointness of $J\cM_0$ in Case (II) see
\cite[Theorems~2.2.7 and 2.6.6]{tretbook}.  The other assertions in this case
are straightforward.

Now assume that Assumption~2.1.(I) is satisfied.  We have
\[
  J\cM_0 = \begin{pmatrix} A & B \\ B^* & -D \end{pmatrix}
\]
which, by \cite[Theorem~2.3.6]{tretbook}, is essentially self-adjoint with
\[
  \dom(\overline{J\cM_0}) = \bigg\{\binom{x}{y}:
  y\in\dom(D),\;x + \overline{(A-\nu)^{-1}B}y\in\dom(A)\bigg\}
\]
and
\[
  \overline{J\cM_0}\binom{x}{y}
  = \begin{pmatrix}
    (A-\nu)\bigl(x+\overline{(A-\nu)^{-1}B}y\bigr) + \nu x \\[1ex]
    B^*x - Dy
  \end{pmatrix}
\]
where $\nu<\min\sigma(A)$ is arbitrary.
Since $J$ is an involution, $\cM_0$ is closable and
$\cM=\ov{\cM_0}=\ov{JJ\cM_0}=J\ov{J\cM_0}$, which shows \eqref{top1} and \eqref{top2}.

It follows also from \cite[Theorem~2.3.6]{tretbook} that $(x,y)^T\in\dom(\cM)$
implies that $x\in\dom(\fra)$.

In order to show \eqref{connection_T_form}, let
$(x,y)^T\in\dom(\cM)$ and $(\hat x,\hat y)^T\in\dom(\fra)\times\cH_2$.
From \eqref{top2} and \eqref{rel_ABstar2} we obtain
\begin{align*}
  & \biggl\langle \cM\binom{x}{y},\binom{\hat x}{\hat y}\biggr\rangle
  \\[1ex]
  &= \Bigl\langle (A-\nu)\bigl(x+(A-\nu)^{-\frac{1}{2}}\ov{(A-\nu)^{-\frac{1}{2}}B}y\bigr)+\nu x,\hat x\Bigr\rangle
  + \bigl\langle -B^*x+Dy,\hat y\bigr\rangle
  \\[1ex]
  &= \Bigl\langle (A-\nu)^{\frac{1}{2}}x+\ov{(A-\nu)^{-\frac{1}{2}}B}y,(A-\nu)^{\frac{1}{2}}\hat x\Bigr\rangle
  + \nu\langle x,\hat x\rangle - \langle B^*x,\hat y\rangle + \langle Dy,\hat y\rangle
  \displaybreak[0]\\[1ex]
  &= (\fra-\nu)[x,\hat x]+\bigl\langle y,B^*(A-\nu)^{-\frac{1}{2}}(A-\nu)^{\frac{1}{2}}\hat x\bigr\rangle
  + \nu\langle x,\hat x\rangle - \langle B^*x,\hat y\rangle + \langle Dy,\hat y\rangle
  \\[1ex]
  &= \fra[x,\hat x] + \langle y,B^*\hat x\rangle
  - \langle B^*x,\hat y\rangle + \langle Dy,\hat y\rangle,
\end{align*}
which proves \eqref{connection_T_form}.
\end{proof}

From \eqref{connection_T_form} we can deduce the following:
if $(x,y)^T\in\dom(\cM)$ and $\cM\binom{x}{y}=\binom{u}{v}$, then
\begin{align}
  \fra[x] + \langle y,B^*x\rangle &= \langle u,x\rangle,
  \label{form1} \\[0.5ex]
  -\langle B^*x,y\rangle + \frd[y] &= \langle v,y\rangle;
  \label{form2}
\end{align}
this follows by setting $(\hat x,\hat y)^T=(x,0)^T$ and $(\hat x,\hat y)^T=(0,y)^T$,
respectively, in \eqref{connection_T_form}.

\begin{remark}
If we introduce the inner product
\begin{equation}\label{indef_pr}
  \biggl[\binom{x}{y},\binom{\hat x}{\hat y}\biggr]
  = \biggl\langle J\binom{x}{y},\binom{\hat x}{\hat y}\biggr\rangle
  = \langle x,\hat x\rangle - \langle y,\hat y\rangle,
  \qquad \binom{x}{y},\binom{\hat x}{\hat y}\in\cH_1\oplus\cH_2,
\end{equation}
with $J$ from \eqref{defJ},
then $\cH_1\oplus\cH_2$ becomes a Krein space with fundamental symmetry $J$,
and $\cM$ is self-adjoint in this Krein space.
This implies that $\sigma(\cM)$ is symmetric with respect to the real axis;
see, e.g.\ \cite[Corollary~VI.6.3]{Bognar}.
We come back to the properties of $\cM$ in the Krein space in Section~\ref{krein}.
For basic properties of Krein spaces see, e.g.\ \cite{Bognar}.
\end{remark}

We can also describe the adjoint of the operator $\cM$ in the Hilbert space $\cH$.

\begin{corollary}\label{adjcor}
The adjoint $\cM^*$ of the operator $\cM$ from Theorem~{\rm\ref{jsaness}}
is equal to the closure of the operator
\[
  \begin{pmatrix} A & -B \\ B^* & D \end{pmatrix}
\]
with domain $\bigl(\dom(A)\cap\dom(B^*)\bigr)\times\bigl(\dom(B)\cap\dom(D)\bigr)$;
the operator $\cM^*$ is given explicitly by
\begin{equation}\label{Tadj}
\begin{aligned}
  \dom(\cM^*) &= \bigg\{\binom{x}{y}: y\in\dom(D),\;x-\overline{(A-\nu)^{-1}B}y\in\dom(A)\bigg\},
  \\[1ex]
  \cM^*\binom{x}{y}
  &= \begin{pmatrix}
    (A-\nu)\bigl(x-\overline{(A-\nu)^{-1}B}y\bigr) + \nu x \\[1ex]
    B^*x + Dy
  \end{pmatrix},
  \qquad \binom{x}{y}\in\dom(\cM^*).
\end{aligned}
\end{equation}
\end{corollary}

\begin{proof}
We have
\[
  \overline{J\cM_0} = (\overline{J\cM_0})^*=(J\cM_0)^*=\cM^*J^*=\cM^*J
\]
and hence $\cM^* = \overline{J\cM_0}J$.  Therefore \eqref{Tadj} holds.
\end{proof}

\section{The Schur complement}\label{schur}

\noindent
In this section we define and study the (first) Schur complement $S$ of the
block operator matrix $\cM$, which is an operator function acting
in the first component $\cH_1$.  Formally, $S$ is given by
\[
  S(z) = A-z+B(D-z)^{-1}B^*, \qquad z\in\rho(D).
\]
However, the domain of $S(z)$ may be too small, and therefore we define $S(z)$
via quadratic forms and for $z$ in a (possibly) smaller set.
The main result of this section is a spectral equivalence of
the operator $\cM$ and the operator function $S$, which is explained further below.

\begin{definition}\label{def:schur}
Let $b_0$ be as in \eqref{defb0} and set
\[
  \mathbb{U} \defeq \bigl\{z\in\mathbb{C}:\dist(z,\sigma(D))>b_0\bigr\}.
\]
Moreover, define the family of sesquilinear forms
\begin{multline*}
  \frs(z)[x,y] \defeq \fra[x,y] - z\langle x,y\rangle +
  \bigl\langle(D-z)^{-1}B^*x,B^*y\bigr\rangle, \\[1ex]
  z\in\mathbb{U},\;x,y\in\dom(\frs(z))\defeq\dom(\fra).
\end{multline*}
\end{definition}

\begin{lemma}\label{forms}
Suppose that Assumption~{\rm\ref{assump}} is satisfied.
Then $\frs(\cdot)$ is a holomorphic family of type {\rm(a)}, i.e.\
$\dom(\frs(z))$ is independent of $z$, $\frs(z)$ is sectorial and closed
for every $z\in\UU$, and $\frs(\cdot)[x]$ is holomorphic on $\UU$ for
every $x\in\dom(\fra)$.
\end{lemma}

\begin{proof}
Evidently, for any $x\in\dom(\fra)$, the function $\frs(\cdot)[x]:\mathbb{U}\to\mathbb{C}$
is holomorphic.  We must show that $\frs(z)$ is closed and sectorial for every $z\in\UU$.
Let $z\in\UU$; then there exist $a\in\RR$, $b\ge0$ such that \eqref{domcons} and
\begin{equation}\label{b0bdist}
  b_0 < b < \dist(z,\sigma(D))
\end{equation}
hold.
For $x\in\dom(\fra)$ we obtain from \eqref{domcons} that
\begin{equation}\label{est_DB}
\begin{aligned}
  \bigl|\bigl\langle(D-z)^{-1}B^*x,B^*x\bigr\rangle\bigr|
  &\le \frac{\|B^*x\|^2}{\dist(z,\sigma(D))} \\
  &\le \frac{a}{\dist(z,\sigma(D))}\|x\|^2+\frac{b}{\dist(z,\sigma(D))}\fra[x].
\end{aligned}
\end{equation}
This, together with \eqref{b0bdist}, implies that $\langle(D-z)^{-1}B^*\cdot\,,B^*\cdot\,\rangle$
is relatively bounded with respect to $\fra$ with relative bound less than one.
Hence $\frs(z)$ is closed and sectorial by \cite[Theorem~VI.1.33]{katopert}.
\end{proof}

It follows from Lemma~\ref{forms} and \cite[Theorem~VI.2.7]{katopert} that,
for each $z\in\mathbb{U}$, there corresponds an m-sectorial operator $S(z)$
to the form $\frs(z)$ in the sense that
\begin{equation}\label{defS}
  \frs(z)[x,y] = \langle S(z)x,y\rangle, \qquad x\in\dom(S(z))\subset\dom(\fra),\;\;
  y\in\dom(\fra).
\end{equation}
The family $S(\cdot)$ is called the \emph{Schur complement} of $\cM$ and is a
holomorphic family of type (B); see \cite[Theorem~VI.4.2]{katopert}.
The \emph{spectrum}, \emph{essential spectrum}, \emph{point spectrum}
and \emph{resolvent set} of the Schur complement are defined as follows:
\begin{alignat*}{2}
  \sigma(S) &\defeq \big\{z\in\mathbb{U}: 0\in\sigma(S(z))\big\}, \qquad &
  \sigmaess(S) &\defeq \big\{z\in\mathbb{U}: 0\in\sigmaess(S(z))\big\}, \\[0.5ex]
  \sigmap(S) &\defeq \big\{z\in\mathbb{U}:0\in\sigmap(S(z))\big\}, \qquad &
  \rho(S) &\defeq \big\{z\in\mathbb{U}:0\in\rho(S(z))\big\}.
\end{alignat*}

\medskip

In the next proposition we describe the domain and the action of $S(z)$ explicitly.

\begin{proposition}\label{schurdom}
Suppose that Assumption~{\rm\ref{assump}} is satisfied,
let $S$ be the Schur complement from \eqref{defS} and let $\nu<\min\sigma(A)$.
For $z\in\UU$ we have
\begin{align*}
  \dom(S(z)) &= \bigl\{x\in\dom(\fra): x + \overline{(A-\nu)^{-1}B}(D-z)^{-1}B^*x\in\dom(A)\bigr\}
  \\[1ex]
  S(z)x &= (A-\nu)\bigl(x+\overline{(A-\nu)^{-1}B}(D-z)^{-1}B^*x\bigr) + (\nu-z)x,
  \\[1ex]
  &\hspace*{50ex} x\in\dom(S(z)).
\end{align*}
\end{proposition}

\begin{proof}
For $x,y\in\dom(\fra)$ we obtain from \eqref{rel_ABstar1} and \eqref{rel_ABstar2} that
\begin{align}
  &\bigl\langle(D-z)^{-1}B^*x,B^*y\bigr\rangle
  = \bigl\langle(D-z)^{-1}B^*x,B^*(A-\nu)^{-\frac{1}{2}}(A-\nu)^{\frac{1}{2}}y\bigr\rangle
  \notag\\
  &= \bigl\langle\overline{(A-\nu)^{-\frac{1}{2}}B}(D-z)^{-1}B^*x,
  (A-\nu)^{\frac{1}{2}}y\bigr\rangle
  \notag\\
  &= \bigl\langle(A-\nu)^{\frac{1}{2}}\ov{(A-\nu)^{-1}B}(D-z)^{-1}B^*x,
  (A-\nu)^{\frac{1}{2}}y\bigr\rangle.
  \label{301}
\end{align}
Now let $x\in\dom(S(z))$ and $y\in\dom(\fra)$.  Then
\begin{align*}
  &\langle S(z)x,y\rangle = \frs(z)[x,y]
  = \fra[x,y] -z\langle x,y\rangle + \bigl\langle(D-z)^{-1}B^*x,B^*y\bigr\rangle
  \\[0.5ex]
  &= \bigl\langle(A-\nu)^{\frac{1}{2}}x,(A-\nu)^{\frac{1}{2}}y\bigr\rangle
  + (\nu-z)\langle x,y\rangle
  \\[0.5ex]
  &\quad + \bigl\langle(A-\nu)^{\frac{1}{2}}\ov{(A-\nu)^{-1}B}(D-z)^{-1}B^*x,
  (A-\nu)^{\frac{1}{2}}y\bigr\rangle
  \\[0.5ex]
  &= \bigl\langle (A-\nu)^{\frac12}\bigl[x+\ov{(A-\nu)^{-1}B}(D-z)^{-1}B^*x\bigr],
  (A-\nu)^{\frac12}y\bigr\rangle + (\nu-z)\langle x,y\rangle.
\end{align*}
It follows from \cite[Theorem~VI.2.1]{katopert} that
$x + \overline{(A-\nu)^{-1}B}(D-z)^{-1}B^*x\in\dom(A)$ and
\[
  (A-\nu)\bigl(x+\overline{(A-\nu)^{-1}B}(D-z)^{-1}B^*x\bigr) = S(z)x - (\nu-z)x.
\]
Conversely, suppose that $x\in\dom(\fra)$ with $x + \ov{(A-\nu)^{-1}B}(D-z)^{-1}B^*x\in\dom(A)$.
Then, for $y\in\dom(\fra)$, we obtain from \eqref{301} that
\begin{align*}
  & \frs(z)[x,y] = \fra[x,y] - z\langle x,y\rangle + \bigl\langle(D-z)^{-1}B^*x,B^*y\bigr\rangle \\
  &= \bigl\langle(A-\nu)^{\frac12}\bigl[x+\ov{(A-\nu)^{-1}B}(D-z)^{-1}B^*x\bigr],
  (A-\nu)^{\frac12}y\bigr\rangle + (\nu-z)\langle x,y\rangle \\
  &= \bigl\langle(A-\nu)\bigl[x+\overline{(A-\nu)^{-1}B}(D-z)^{-1}B^*x\bigr],y\rangle
  + (\nu - z)\langle x,y\rangle.
\end{align*}
Now \cite[Theorem~VI.2.1\,(iii)]{katopert} implies that $x\in\dom(S(z))$.
\end{proof}

The next lemma gives a first connection between the operator $\cM$ and the
Schur complement $S$.

\begin{lemma}\label{le:Sz_T_z}
Let $z\in\UU$.
\begin{myenum}
\item
If $x\in\dom(S(z))$, then
\[
  \begin{pmatrix} x \\[0.5ex] (D-z)^{-1}B^*x \end{pmatrix} \in \dom(\cM)
\]
and
\[
  (\cM-z)\begin{pmatrix} x \\[0.5ex] (D-z)^{-1}B^*x \end{pmatrix}
  = \begin{pmatrix} S(z)x \\[0.5ex] 0 \end{pmatrix}.
\]
\item
If\, $(x,y)^T\in\dom(\cM)$ and
\[
  (\cM-z)\binom{x}{y} = \binom{u}{0}
\]
with some $u\in\cH_1$, then $x\in\dom(S(z))$, $S(z)x=u$ and $y=(D-z)^{-1}B^*x$.
\end{myenum}
\end{lemma}

\begin{proof}
(i)
Let $x\in\dom(S(z))$ and set $y\defeq(D-z)^{-1}B^*x$.
Then $x\in\dom(\frs(z))=\dom(\fra)$ and hence $x\in\dom(B^*)$.
Moreover, $y\in\dom(D)$.  Now, combining Theorem~\ref{jsaness} and Proposition~\ref{schurdom}
we obtain that $(x,y)^T\in\dom(\cM)$ and
\begin{align*}
  (\cM-z)\binom{x}{y}
  &= \begin{pmatrix}
    (A-\nu)\bigl(x+\ov{(A-\nu)^{-1}B}(D-z)^{-1}B^*x\bigr) + \nu x - zx \\[1ex]
    -B^*x + (D-z)(D-z)^{-1}B^*x
  \end{pmatrix}
  \\[0.5ex]
  &= \begin{pmatrix} S(z)x \\[0.5ex] 0 \end{pmatrix}.
\end{align*}

(ii)
The assumption implies that $y=(D-z)^{-1}B^*x$.  The claim follows again
from Theorem~\ref{jsaness} and Proposition~\ref{schurdom}.
\end{proof}

Before we prove the spectral equivalence of $\cM$ and $S$, we need a lemma
about approximative eigensequences, which is also used in later sections.

\begin{lemma}\label{le:approx_es}
Let $z\in\CC$ and let $(x_n,y_n)^T\in\dom(\cM)$, $n\in\NN$, such that
\[
  \|x_n\|^2+\|y_n\|^2 = 1 \qquad\text{and}\qquad
  (\cM-z)\binom{x_n}{y_n} \to 0 \quad\text{as }n\to\infty.
\]
Then the following statements hold.
\begin{myenum}
\item
The sequences $\fra[x_n]$ and $\|B^*x_n\|$ are bounded.
\item
If\, $z\in\rho(D)$, then
\begin{equation}\label{yn}
  y_n = (D-z)^{-1}B^*x_n + w_n \qquad\text{with}\quad w_n\to0.
\end{equation}
\item
If\, $z\in\UU$, then
\[
  \liminf_{n\to\infty} \|x_n\| > 0.
\]
Moreover, if\, $\xi_n\in\dom(B^*)$, $n\in\NN$, are such that $(\xi_n)$ and $(B^*\xi_n)$
are bounded sequences, then
\[
  \lim_{n\to\infty} \frs(z)[x_n,\xi_n] = 0.
\]
In particular,
\begin{equation}\label{sxnto0}
  \lim_{n\to\infty} \frs(z)[x_n] = 0.
\end{equation}
\item
If\, $z\in\UU$ and $x_n\to x_0$ for some $x_0\in\cH_1$, then
\[
  x_0\in\dom(\fra) \qquad\text{and}\qquad
  B^*x_n\to B^*x_0.
\]
\end{myenum}
\end{lemma}

\begin{proof}
For the first items we may assume that $(x_n,y_n)^T$ is only a bounded sequence
rather than a normalised one.  This is used in the the proof of item (iv).

(i)
Set
\begin{equation}\label{unvn_M_z}
  \binom{u_n}{v_n} \defeq (\cM-z)\binom{x_n}{y_n}
  = \begin{pmatrix}
    (A-\nu)\bigl(x_n+\ov{(A-\nu)^{-1}B}y_n\bigr)+(\nu-z)x_n \\[1ex]
    -B^*x_n+(D-z)y_n
  \end{pmatrix}.
\end{equation}
From \eqref{form1} we obtain
\begin{equation}\label{form11}
  \fra[x_n] - z\|x_n\|^2 + \langle y_n,B^*x_n\rangle = \langle u_n,x_n\rangle \to 0.
\end{equation}
This, together with \eqref{domcons} implies that, as $n\to\infty$,
\begin{align*}
  \fra[x_n] &= -\langle y_n,B^*x_n\rangle + \cO(1) \\[0.5ex]
  &\le \|y_n\|\,\|B^*x_n\| + \cO(1) \\[0.5ex]
  &\le \|y_n\|\sqrt{b\fra[x_n]+a\|x_n\|^2} + \cO(1).
\end{align*}
It follows that $\fra[x_n]$ is bounded and, again by \eqref{domcons}, that also
$\|B^*x_n\|$ is bounded.

(ii)
Let $z\in\rho(D)$.
Comparing the second components in \eqref{unvn_M_z} we obtain
\begin{equation}\label{yn1}
  y_n = (D-z)^{-1}B^*x_n + (D-z)^{-1}v_n,
\end{equation}
which implies \eqref{yn}.

(iii)
Let $z\in\UU$ and let $\xi_n$ be as in the statement of the lemma.
Relations~\eqref{connection_T_form} and \eqref{yn1} yield
\begin{align*}
  &\biggl\langle(\cM-z)\binom{x_n}{y_n},\binom{\xi_n}{v_n}\biggr\rangle
  \\[0.5ex]
  &= \fra[x_n,\xi_n] - z\langle x_n,\xi_n\rangle + \langle y_n,B^*\xi_n\rangle
  - \langle B^*x_n,v_n\rangle + \bigl\langle(D-z)y_n,v_n\bigr\rangle
  \\[0.5ex]
  &= \fra[x_n,\xi_n] - z\langle x_n,\xi_n\rangle + \bigl\langle(D-z)^{-1}B^*x_n,B^*\xi_n\bigr\rangle
  + \bigl\langle(D-z)^{-1}v_n,B^*\xi_n\bigr\rangle
  \\[0.5ex]
  &\quad - \langle B^*x_n,v_n\rangle + \langle B^*x_n,v_n\rangle + \langle v_n,v_n\rangle
  \\[0.5ex]
  &= \frs(z)[x_n,\xi_n] + \bigl\langle(D-z)^{-1}v_n,B^*\xi_n\bigr\rangle + \|v_n\|^2.
\end{align*}
The left-hand side and the second and the third terms on the right-hand side converge to $0$
by the assumption on $\xi_n$.  Hence $\frs(z)[x_n,\xi_n]\to0$.

For $\xi_n=x_n$ the assumptions on $\xi_n$ are satisfied because of item (i);
hence $\frs(z)[x_n]\to0$.
Note that this remains true if $(x_n,y_n)^T$ is just bounded.

Before we prove the remaining items, let us show the following inequalities.
Let $a\in\RR$, $b\ge0$ such that \eqref{domcons} and \eqref{b0bdist} hold
and let $\alpha_-$ be as in \eqref{defalmidepl}.
For $x\in\dom(\fra)$ we obtain from \eqref{est_DB} that
\begin{align*}
  \fra[x] &\le \bigl|\fra[x]\bigr|
  = \bigl|\frs(z)[x]+z\|x\|^2-\bigl\langle(D-z)^{-1}B^*x,B^*x\bigr\rangle\bigr|
  \\[0.5ex]
  &\le \bigl|\frs(z)[x]\bigr| + |z|\,\|x\|^2 + \bigl|\bigl\langle(D-z)^{-1}B^*x,B^*x\bigr\rangle\bigr|
  \\[0.5ex]
  &\le \bigl|\frs(z)[x]\bigr| + |z|\,\|x\|^2
  + \frac{b\fra[x]}{\dist(z,\sigma(D))} + \frac{a\|x\|^2}{\dist(z,\sigma(D))}
\end{align*}
and hence
\begin{equation}\label{inequs}
\begin{aligned}
  \alpha_-\frac{\dist(z,\sigma(D))-b}{\dist(z,\sigma(D))}\|x\|^2
  &\le \frac{\dist(z,\sigma(D))-b}{\dist(z,\sigma(D))}\fra[x] \\[0.5ex]
  &\le \bigl|\frs(z)[x]\bigr| + \biggl(|z|+\frac{a}{\dist(z,\sigma(D))}\biggr)\|x\|^2
\end{aligned}
\end{equation}
In the following assume that $\|x_n\|^2+\|y_n\|^2=1$.
Next we show the first statement of (iii), i.e.\ that $\liminf_{n\to\infty}\|x_n\|>0$.
Suppose to the contrary that there exists
a subsequence $(x_{n_k})$ of $(x_n)$ such that $x_{n_k}\to0$.
Then the left-hand and the right-hand sides of \eqref{inequs} with $x=x_{n_k}$
converge to $0$ as $k\to\infty$ by the already proved relation \eqref{sxnto0}.
Hence $\fra[x_{n_k}]\to0$ since $\dist(z,\sigma(D))>b$.
From \eqref{domcons} we obtain that $(D-z)^{-1}B^*x_{n_k}\to0$, which is a
contradiction to \eqref{yn} and the relation $\|y_{n_k}\|\to1$.

(iv)
Assume that $x_n\to x_0$.  It follows from the already proved items,
applied to $x_n-x_m$ instead of $x_n$,
that $\frs(z)[x_n-x_m]\to0$ as $n,m\to\infty$.
Hence the left-hand and the right-hand sides of \eqref{inequs} with $x=x_n-x_m$
converge to $0$ as $n,m\to\infty$, and therefore also $\fra[x_n-x_m]\to0$.
This means that $x_n\stackrel{\fra}\longrightarrow x_0$ (see \cite[\S VI.3]{katopert}),
which implies that $x_0\in\dom(\fra)$ because $\fra$ is closed.
Again by \eqref{domcons} we obtain that $\|B^*x_n-B^*x_m\|\to0$ as $n,m\to\infty$.
Since $B^*$ is closed, it follows that $B^*x_n\to B^*x_0$.
\end{proof}

The theorem below is analogous to \cite[Proposition~2.2]{KLT04} which treats the
self-adjoint case.  The last part of our proof is more involved in the sense that it
uses Lemma~\ref{le:approx_es}.  This is due to the loss of
self-adjointness and the possibility of non-real spectrum.
See also \cite[Propositions~2.7 and 2.8]{JT02} for a similar result under the
assumption that $B^*$ is $A$-form-compact.

\begin{theorem}\label{specequiv}
Suppose that Assumption~{\rm\ref{assump}} is satisfied, let $\cM$ be the operator
as in Theorem~{\rm\ref{jsaness}} and let $S$ be its Schur complement as in \eqref{defS}.
Then the following relations hold:
\begin{align}
  & \sigma(S) = \sigma(\cM)\cap\UU, \qquad
  \sigmap(S) = \sigmap(\cM)\cap\UU,
  \label{sigmaST}\\[1ex]
  &\! \nul(S(z)) = \nul(\cM - z) \qquad\text{for}\;\;z\in\UU.
  \label{nul_eq}
\end{align}
\end{theorem}

\begin{proof}
First we show \eqref{nul_eq}.
Let $z\in\UU$ and $(x,y)^T\in\ker(\cM-z)$.  It follows from Lemma~\ref{le:Sz_T_z}\,(ii)
that $x\in\ker(S(z))$.  Hence $\nul(\cM-z)\le\nul(S(z))$.

Now let $x\in\ker(S(z))$.  Lemma~\ref{le:Sz_T_z}\,(i) implies that
\[
  \begin{pmatrix} x \\[0.5ex] (D-z)^{-1}B^*x \end{pmatrix} \in \ker(\cM-z).
\]
Therefore $\nul(S(z))\le\nul(\cM-z)$, and \eqref{nul_eq} is proved.
From this we also obtain the second relation in \eqref{sigmaST}.

It remains to show the first relation in \eqref{sigmaST}.
Let $z\in\rho(\cM)\cap\mathbb{U}$ and $u\in\mathcal{H}_1$.  Then there
exists an $(x,y)^T\in\dom(\cM)$ with
\[
  (\cM - z)\binom{x}{y} = \binom{u}{0}.
\]
It follows from Lemma~\ref{le:Sz_T_z}\,(ii) that $x\in\dom(S(z))$ and $S(z)x=u$.
Hence $S(z)$ is surjective.  By the already proved relation in \eqref{nul_eq}
we obtain that $z\in\rho(S(z))$.  Hence $\sigma(S)\subset\sigma(\cM)\cap\UU$.

Now let $z\in\rho(S)$.  Then $\cM-z$ is injective and therefore has an inverse.
A direct calculation establishes that this inverse,
restricted to $\mathcal{H}_1\times\dom(B(D-z)^{-1})$, is given by
\begin{equation}\label{invrep}
  (\cM - z)^{-1} = \begin{pmatrix}
    S(z)^{-1}& -S(z)^{-1}B(D - I)^{-1} \\[1ex]
    F(z) & (D - z)^{-1} - F(z)B(D - I)^{-1}
  \end{pmatrix}
\end{equation}
where $F(z) \defeq (D - z)^{-1}B^*S(z)^{-1} $.
The set $\mathcal{H}_1\times\dom(B(D-z)^{-1})$
is dense in $\mathcal{H}_1\times\mathcal{H}_2$:
if Assumption~\ref{assump}.(I) is satisfied, this follows from the fact
that $\dom(B)$ is a core for $D$;
if Assumption~\ref{assump}.(II) is satisfied, then $\dom(B(D-z)^{-1})=\cH_2$.
It therefore suffices to show that the operator on the right-hand side
of \eqref{invrep} is bounded.  We suppose the contrary.
Then there exists a sequence
\begin{equation}\label{opeq}
  \binom{x_n}{y_n}\in\dom(\cM) \qquad\text{with}\quad
  (\cM-z)\binom{x_n}{y_n}
  \eqdef \binom{u_n}{v_n} \to 0,
\end{equation}
$\|x_n\|^2 + \|y_n\|^2 = 1$, $v_n\in \dom(B(D-z)^{-1})$ and hence
\begin{equation}\label{reseq}
  \binom{x_n}{y_n} = \begin{pmatrix}
    S(z)^{-1}u_n - S(z)^{-1}B(D-z)^{-1}v_n \\[1ex]
    F(z)u_n + (D-z)^{-1}v_n - F(z)B(D-z)^{-1}v_n
  \end{pmatrix}.
\end{equation}
From \cite[Theorem~VI.2.5]{katopert} we have $S(z)^* = S(\overline{z})$ and
therefore $\overline{z}\in\rho(S)$.
Further, $\dom(S(\overline{z}))\subset\dom(\fra)\subset\dom(B^*)$
and hence the operator $B^*S(\overline{z})^{-1}$ is bounded.
Since $(S(z)^{-1}B)^*= B^*S(\overline{z})^{-1}$, it follows that $S(z)^{-1}B$ is bounded.
Then, using \eqref{reseq}, we deduce that $x_n\to 0$, which is a contradiction
to Lemma~\ref{le:approx_es}\,(iii).
Hence $\sigma(\cM)\cap\mathbb{U}\subset\sigma(S)$.
\end{proof}

In Theorem~\ref{th:sess_equiv} below we show the equivalence of essential spectra
of $S$ and $\cM$ in a certain interval.

In the next proposition we consider the situation where we can
describe the essential spectrum of $\cM$.

\begin{proposition}\label{pr:sessALMS}
Suppose that Assumption~{\rm\ref{assump}.(I)} is satisfied and that $A$ has
compact resolvent.  Then
\begin{equation}\label{sessALMS}
  \sigmaess(\cM) = \sigmaess\bigl(D+\ov{B^*(A-\nu)^{-1}B}\bigr)
  \subset \bigl[\inf\sigmaess(D),\sup\sigmaess(D)+b_0\bigr]
\end{equation}
for any $\nu<\min\sigma(A)$.
\end{proposition}

\begin{proof}
Since
\[
  \ov{B^*(A-\nu)^{-1}B} = B^*(A-\nu)^{-\frac12}\ov{(A-\nu)^{-\frac12}B}
\]
is bounded by \eqref{rel_ABstar2} and its preceding paragraph and since
\[
  \ov{B^*(A-\nu)^{-2}B} = B^*(A-\nu)^{-\frac12}(A-\nu)^{-1}\ov{(A-\nu)^{-\frac12}B}
\]
is compact, it follows that all assumptions of \cite[Theorem~2.2]{ALMS94} are satisfied.
The latter yields the first equality in \eqref{sessALMS}.
Note that the essential spectrum of $D+\ov{B^*(A-\nu)^{-1}B}$ is independent of $\nu$
since differences of these operators for different $\nu$ are compact.

To show the inclusion in \eqref{sessALMS}, let $a\in\RR$ and $b\ge0$ be any pair
of numbers such that \eqref{domcons} holds.
Since $\ov{B^*(A-\nu)^{-1}B}\ge0$, we have
\begin{equation}\label{sessALMS2}
  \sigmaess\bigl(D+\ov{B^*(A-\nu)^{-1}B}\bigr)
  \subset \bigl[\inf\sigmaess(D),\sup\sigmaess(D)+\bigl\|\ov{B^*(A-\nu)^{-1}B}\bigr\|\,\bigr]
\end{equation}
for any $\nu<\min\sigma(A)$.
Moreover, if $\nu<0$, $\nu<\min\sigma(A)$ and $x\in\cH_1$,
we obtain from \eqref{domcons} that
\begin{align*}
  \bigl\|B^*(A-\nu)^{-\frac12}x\bigr\|^2
  &\le a\bigl\|(A-\nu)^{-\frac12}x\bigr\|^2 + b\fra\bigl[(A-\nu)^{-\frac12}x\bigr]
  \\[0.5ex]
  &\le a\bigl\|(A-\nu)^{-\frac12}x\bigr\|^2 + b(\fra-\nu)\bigl[(A-\nu)^{-\frac12}x\bigr]
  \\[0.5ex]
  &= a\bigl\|(A-\nu)^{-\frac12}x\bigr\|^2 + b\|x\|^2.
\end{align*}
This implies that
\[
  \liminf_{\nu\to-\infty}\, \bigl\|\ov{B^*(A-\nu)^{-1}B}\bigr\|
  = \liminf_{\nu\to-\infty}\, \bigl\|B^*(A-\nu)^{-\frac12}\bigr\|^2 \le b.
\]
If we take the infimum over all $b>b_0$ and combine this relation with \eqref{sessALMS2},
we obtain the inclusion in \eqref{sessALMS}.
\end{proof}

\begin{remark}\label{re:comp2}
If, in addition to the assumptions of Proposition~\ref{pr:sessALMS},
the operator $B^*(A-\nu)^{-\frac12}$ is compact for some $\nu<\min\sigma(A)$,
then $\ov{B^*(A-\nu)^{-1}B}$ is compact as well, and
hence $\sigmaess(\cM)=\sigmaess(D)$.
\end{remark}

\section{The quadratic numerical range}\label{numrange_sec}

\noindent
The quadratic numerical range of a block operator matrix is a very useful tool
for proving spectral enclosures, it uses the block structure of the operator, and
the enclosures are tighter than those obtained from the numerical range.
It was introduced in \cite{LT98} and later studied in various paper; see, e.g.\
\cite{LMMT01,LLT02,LLMT05,T09,tretbook}.

\begin{definition}\label{def:W2}
Suppose that Assumption~{\rm\ref{assump}} is satisfied and let $\cM$ be the operator
as in Theorem~{\rm\ref{jsaness}}.
The \emph{quadratic numerical range} of $\cM$, denoted by $W^2(\cM)$, is defined as
the set of eigenvalues of all $2\!\times\!2$-matrices
\[
  \cM_{x,y} \defeq \begin{pmatrix}
    \dfrac{\fra[x]}{\|x\|^2} & \dfrac{\langle y,B^*x\rangle}{\|x\|\,\|y\|} \\[3ex]
    -\dfrac{\langle B^*x,y \rangle}{\|x\|\,\|y\|} & \dfrac{\frd[y]}{\|y\|^2}
  \end{pmatrix}, \qquad
  x\in\dom(\fra)\backslash\{0\},\;\; y\in\dom(\frd)\backslash\{0\},
\]
i.e.\
\[
  W^2(\cM) \defeq \bigl\{z\in\CC:\exists\,x\in\dom(\fra)\backslash\{0\},\,
  y\in\dom(\frd)\backslash\{0\} \;\text{such that}\; z\in\sigma(\cM_{x,y})\bigr\}.
\]
The eigenvalues of $\cM_{x,y}$ are
\[
  \lambda_\pm\binom{x}{y}
  \defeq \frac{1}{2}\Bigg(\frac{\fra[x]}{\|x\|^2} + \frac{\frd[y]}{\|y\|^2}
  \pm\sqrt{\bigg(\frac{\fra[x]}{\|x\|^2} - \frac{\frd[y]}{\|y\|^2}\bigg)^2 - 4\frac{|\langle y,B^*x\rangle|^2}{\|x\|^2\,\|y\|^2}}~\Bigg).
\]
\end{definition}

\pagebreak[3]

\begin{remark}
\rule{0ex}{1ex}
\begin{myenum}
\item
Note that our definition differs slightly from that in \cite{tretbook},
where $x$ and $y$ vary only in $\dom(A)$ and $\dom(D)$, respectively.
However, in order to have $\lambda_\pm\binom{x}{y}$ defined for all $(x,y)^T\in\dom(\cM)$
with $x,y\ne0$, we chose the larger sets $\dom(\fra)$ and $\dom(\frd)$.
These sets were also used in \cite{KLT04} for self-adjoint block operator matrices.
\item
It is easy to see that $W^2(\cM)$ is symmetric with respect to the real axis
and it consists of at most two connected components.
It follows in the same way as in \cite[Proposition~2.3]{LLMT05} that
if $\dim\cH>2$ and $W^2(\cM)$ contains at least one non-real point, then
$W^2(\cM)$ is connected.
\end{myenum}
\end{remark}

\smallskip

\noindent
We shall often use the following notation.
Let $x\in\dom(\fra)\backslash\{0\}$ and $y\in\dom(\frd)\backslash\{0\}$ and set
\begin{equation}\label{alphabetadelta}
  \alpha \defeq \frac{\fra[x]}{\|x\|^2}\,, \qquad
  \beta \defeq \frac{\langle y,B^*x\rangle}{\|x\|\,\|y\|}\,, \qquad
  \delta \defeq \frac{\frd[y]}{\|y\|^2}\,.
\end{equation}
Then
\begin{equation}\label{lambdapmalpha}
  \lambda_\pm\binom{x}{y}
  = \frac{\alpha+\delta}{2}\pm\sqrt{\biggl(\frac{\alpha-\delta}{2}\biggr)^2-|\beta|^2\,}.
\end{equation}
It follows from \eqref{domcons} that
\begin{equation}\label{estbeta}
  |\beta|^2 = \frac{|\langle y,B^*x\rangle|^2}{\|x\|^2\|y\|^2}
  \le \frac{\|B^*x\|^2}{\|x\|^2} \le \frac{b\fra[x]+a\|x\|^2}{\|x\|^2}
  = b\alpha+a.
\end{equation}

\bigskip

First we show that $W^2(\cM)$ contains the eigenvalues of $\cM$.

\begin{lemma}\label{qnrps}
Suppose that Assumption~{\rm\ref{assump}} is satisfied.
Then $\sigmap(\cM)\subset W^2(\cM)$.
\end{lemma}

\begin{proof}
Let $z\in\sigmap(\cM)$.  Then there exists a non-zero vector $(x,y)^T\in\ker(\cM-z)$, i.e.\
\begin{align}
  (A-\nu)\bigl(x+\ov{(A-\nu)^{-1}B}y\bigr) + (\nu-z)x &= 0,
  \label{peak1} \\[0.5ex]
  -B^*x + (D-z)y &= 0.
  \label{peak2}
\end{align}
It follows from \eqref{form1} and \eqref{form2} that
\begin{align}
  \fra[x] - z\|x\|^2 + \langle y,B^*x\rangle &= 0,
  \label{peak3} \\[0.5ex]
  -\langle B^*x,y\rangle + \frd[y] - z\|y\|^2 &= 0.
  \label{peak4}
\end{align}
Let us first consider the case when $x=0$.  Then $\frd[y]=z\|y\|^2$ by \eqref{peak4}.
Moreover, \eqref{peak1} implies that $\ov{(A-\nu)^{-1}B}y=0$ and
hence $\ov{(A-\nu)^{-\frac{1}{2}}B}y=0$ by \eqref{rel_ABstar2}.
For any $u\in\dom(\fra)\backslash\{0\}$ we have
\[
  \langle y,B^*u\rangle
  = \bigl\langle y,B^*(A-\nu)^{-\frac{1}{2}}(A-\nu)^{\frac{1}{2}}u\bigr\rangle
  = \bigl\langle \ov{(A-\nu)^{-\frac{1}{2}}B}y,(A-\nu)^{\frac{1}{2}}u\bigr\rangle = 0
\]
and hence
\[
  \cM_{u,y}
  = \begin{pmatrix}
    \dfrac{\fra[u]}{\|u\|^2} & 0 \\[3ex]
    0 & z
  \end{pmatrix},
\]
which shows that $z\in\sigma(\cM_{u,y})\subset W^2(\cM)$.

Next suppose that $y=0$. Then $x\ne 0$, $B^*x=0$ and $(A-z)x=0$
by \eqref{peak1} and \eqref{peak2}.
For any $v\in\dom(\frd)\backslash\{0\}$ we have
\begin{align*}
  \cM_{x,v} = \begin{pmatrix}
    z & 0 \\[1ex]
    0 & \dfrac{\frd[v]}{\|v\|^2}
  \end{pmatrix},
\end{align*}
which yields $z\in\sigma(\cM_{x,v})\subset W^2(\cM)$.

Finally, we assume that $x\ne 0$ and $y\ne 0$.  Then \eqref{peak3} and \eqref{peak4}
imply that
\begin{align*}
  \bigl(\cM_{x,y}-z\bigr) \begin{pmatrix} \|x\| \\[1ex] \|y\| \end{pmatrix}
  &= \begin{pmatrix}
    \dfrac{\fra[x]}{\|x\|^2}-z & \dfrac{\langle y,B^*x\rangle}{\|x\|\,\|y\|} \\[3ex]
    -\dfrac{\langle B^*x,y\rangle}{\|x\|\,\|y\|} & \dfrac{\frd[y]}{\|y\|^2}-z
  \end{pmatrix}
  \begin{pmatrix} \|x\| \\[1ex] \|y\| \end{pmatrix} \\[2ex]
  &= \begin{pmatrix}
    \dfrac{\fra[x]}{\|x\|} - z\|x\| + \dfrac{\langle y,B^*x\rangle}{\|x\|} \\[3ex]
    -\dfrac{\langle B^*x,y\rangle}{\|y\|} + \dfrac{\frd[y]}{\|y\|} - z\|y\|
  \end{pmatrix}
  = 0,
\end{align*}
which shows that $z\in\sigma(\cM_{x,y})\subset W^2(\cM)$.
\end{proof}

The next lemma is shown in the same way as \cite[Proposition~3.3]{T09}.

\begin{lemma}\label{le:qnrnr}
If\, $\dim\cH_1\ge 2$, then $W(D)\subset W^2(\cM)$.
If\, $\dim\cH_2\ge 2$, then $W(A)\subset W^2(\cM)$.
\end{lemma}

In the following definition we introduce a set, $\BB$, in which
the quadratic numerical range and the spectrum of $\cM$ are contained,
as we shall show in Proposition~\ref{numrange} and Theorem~\ref{th:specinB} below.
Moreover, we introduce condition \condA{} under which $W^2(\cM)$ and $\sigma(\cM)$
are contained in $\RR$.
Some comments concerning these definitions are given in Remark~\ref{re:pos_gap};
see also Figure~\ref{fig1}, which shows the set $\BB$ when $D$ is bounded.

\begin{definition}\label{def:Betc}
Assume that Assumption~\ref{assump} is satisfied.
Let $a\in\RR$, $b\ge0$ such that \eqref{domcons} holds and let $\alpha_-$, $\delta_+$
as in \eqref{defalmidepl}.
Moreover, if $D$ is bounded, set $\delta_-\defeq\min\sigma(D)$;
otherwise set $\delta_-\defeq-\infty$.
\begin{myenum}
\item
We say that condition \condA{} is satisfied if
\begin{equation}\label{condA1}
  b\delta_++b^2+a \le 0 \qquad\text{and}\qquad b>0
\end{equation}
or
\begin{equation}\label{condA2}
  \frac{\alpha_--\delta_+}{2} \ge b+\sqrt{\bigl(b\delta_++b^2+a\bigr)_+\!}\;.
\end{equation}
\item
Set
\begin{align}
  \mu \defequ \delta_+ + b + \sqrt{\bigl(b\delta_++b^2+a\bigr)_+\!}\,;
  \label{defmu}
  \\[1ex]
  \xi_1 \defequ \alpha_--\max\biggl\{\frac{b}{2},\sqrt{b\alpha_-+a}\biggr\};
  \label{defxi1}\\[1ex]
  \xi_2 \defequ \begin{cases}
    \dfrac{\alpha_-+\delta_-}{2} & \text{if $D$ is bounded}, \\[2ex]
    -\infty & \text{otherwise};
  \end{cases}
  \label{defxi2}\\[1ex]
  \xi_- \defequ \max\{\xi_1,\xi_2\};
  \label{defximi}\displaybreak[0]\\[1ex]
  \mu_- \defequ \min\{\alpha_-,\delta_-\};
  \\[1ex]
  \mu_+ \defequ \begin{cases}
    -\dfrac{a}{b} & \text{if \eqref{condA1} is satisfied but \eqref{condA2} is not}, \\[2ex]
    \alpha_--b-\sqrt{\bigl(b\delta_++b^2+a\bigr)_+\!}
      & \text{if \eqref{condA2} is satisfied but \eqref{condA1} is not}, \\[2ex]
    \max\Bigl\{-\dfrac{a}{b},\,\alpha_--b\Bigr\}
      & \text{if \eqref{condA1} and \eqref{condA2} are satisfied}; \hspace*{-4ex}
  \end{cases}
  \label{defmupl}\\[1ex]
  \eta \defequ \sqrt{\biggl(b\delta_++b^2+a
  -\biggl(\Bigl(\frac{\alpha_--\delta_+}{2}-b\Bigr)_+\biggr)^{\!2\,}\biggr)_+}\,.
  \label{defeta}
\end{align}
\item
Define the sets
\begin{align}
  \BBnr &\defeq \Bigl\{z\in\CC: \xi_- \le \Re z \le \mu,\;
    |\Im z| \le \eta\Bigr\};
  \\[1ex]
  \BB & \defeq \begin{cases}
    [\mu_-,\mu]\cup[\mu_+,\infty) & \text{if \condA{} is satisfied and $D$ is bounded}, \\[1ex]
    (-\infty,\mu]\cup[\mu_+,\infty) & \text{if \condA{} is satisfied and $D$ is unbounded}, \\[1ex]
    [\mu_-,\infty)\cup \BBnr & \text{if \condA{} is not satisfied and $D$ is bounded}, \\[1ex]
    \RR\cup\BBnr & \text{if \condA{} is not satisfied and $D$ is unbounded}.
  \end{cases}
  \label{defB}
\end{align}
\end{myenum}
\end{definition}

\begin{figure}[h]
\setlength{\unitlength}{1mm}
\begin{picture}(60,20)(-25,-10)
\put(-28,10){(a)}
\put(-30,0){\vector(1,0){60}}
\put(0,0.1){\line(1,0){29}}
\put(0,-0.1){\line(1,0){29}}
\put(-25,0.1){\line(1,0){15}}
\put(-25,-0.1){\line(1,0){15}}
\put(-10,-0.5){\line(0,1){1}}
\put(-11,-3.7){$\mu$}
\put(0,-0.5){\line(0,1){1}}
\put(-1,-3.7){$\mu_+$}
\put(-25,-0.5){\line(0,1){1}}
\put(-26,-3.7){$\mu_-$}
\end{picture}
\hspace*{3mm}
\begin{picture}(60,20)(-25,-10)
\put(-28,10){(b)}
\put(-30,0){\vector(1,0){60}}
\put(-25,0.1){\line(1,0){54}}
\put(-25,-0.1){\line(1,0){54}}
\put(-15,-8){\line(0,1){16}}
\put(10,-8){\line(0,1){16}}
\put(-15,8){\line(1,0){25}}
\put(-15,-8){\line(1,0){25}}
\put(-15,4){\line(1,1){4}}
\put(-15,0){\line(1,1){8}}
\put(-15,-4){\line(1,1){12}}
\put(-15,-8){\line(1,1){16}}
\put(-11,-8){\line(1,1){16}}
\put(-7,-8){\line(1,1){16}}
\put(-3,-8){\line(1,1){13}}
\put(1,-8){\line(1,1){9}}
\put(5,-8){\line(1,1){5}}
\put(-25,-0.5){\line(0,1){1}}
\put(-26,-3.7){$\mu_-$}
\put(11,-3.7){$\mu$}
\put(-18.8,-3.7){$\xi_-$}
\put(11,4){$\eta$}
\end{picture}
\caption{The set $\BB$ when $D$ is bounded; (a) shows the case when \condA{} is satisfied;
(b) shows the case when \condA{} is not satisfied.}
\label{fig1}
\end{figure}
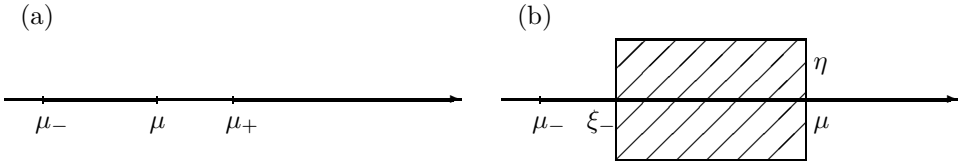

\pagebreak[3]

\begin{remark}\label{re:pos_gap}
\rule{0ex}{1ex}
\begin{myenum}
\item
If $B\ne0$, then $a\ne0$ or $b\ne0$ and therefore the right-hand side
of \eqref{condA2} is positive and $\mu>\delta_+$ (note that $a>0$ if $b=0$).
\item
If $B\ne0$ and the first inequality in \eqref{condA1} is satisfied, then
automatically $b>0$ (since $a>0$ if $b=0$).
\item
Assume that \condA{} is satisfied and that $B\ne0$.
Then
\begin{equation}\label{chain_inequ}
  \delta_+ < \mu \le \mu_+ \le \alpha_-.
\end{equation}
In particular, the spectra of $A$ and $D$ must be separated.
The inequalities in \eqref{chain_inequ} are true because of the following considerations.
If \eqref{condA1} holds, then $b>0$, and from (i), \eqref{condA1} and \eqref{abalmi},
we obtain
\[
  \delta_+ < \mu = \delta_++b \le -\frac{a}{b} \le \mu_+ \le \alpha_-.
\]
If \eqref{condA2} holds, then
\begin{equation}\label{chain_inequ2}
  \delta_+ < \mu \le \frac{\alpha_-+\delta_+}{2} \le \mu_+ \le \alpha_-.
\end{equation}
If \eqref{condA2} holds but \eqref{condA1} does not, then $\mu_+<\alpha_-$.
If the first inequality in \eqref{condA1} or the inequality in \eqref{condA2} is strict,
then $\mu<\mu_+$.  Moreover, if \eqref{condA2} is strict, then
\[
  \delta_+ < \mu < \frac{\alpha_-+\delta_+}{2} < \mu_+ \le \alpha_-.
\]
\item
If \eqref{condA1} is satisfied, it can happen that $\mu_+=\alpha_-$.
Consider, for instance the situation when $\alpha_-=0$ and \eqref{domcons}
holds with $a=0$ and $b>0$.  Then $\mu_+=0$.
On the other hand, if \eqref{condA1} is not satisfied but \eqref{condA2} is,
then always $\mu_+<\alpha_-$.
\item
The number $\mu_+$ can also be characterised as
\[
  \mu_+ = \max\bigl\{\mu_+^{(1)},\mu_+^{(2)}\bigr\}
\]
where
\begin{align*}
  \mu_+^{(1)} &\defeq \begin{cases}
    -\dfrac{a}{b} & \text{if \eqref{condA1} is satisfied}, \\[2ex]
    -\infty & \text{otherwise},
  \end{cases}
  \\[1ex]
  \mu_+^{(2)} &\defeq \begin{cases}
    \alpha_--b-\sqrt{\bigl(b\delta_++b^2+a\bigr)_+\!}
      & \text{if \eqref{condA2} is satisfied}, \\[1ex]
    -\infty & \text{otherwise}.
  \end{cases}
\end{align*}
\item
If $B$ is ``small'', i.e.\ $a$ and $b$ are small, then in general $\xi_1$ gives
the better lower bound for the real part of non-real elements from $\BB$.
If $D$ is bounded and $B$ is ``large'', then $\xi_2$ gives the better bound
as it is independent of $B$.
\item
It is elementary to see that $\eta=0$ if and only if \condA{} is satisfied;
moreover,
\[
  \eta = \begin{cases}
    0 & \text{if \condA{} is satisfied},
    \\[1ex]
    \displaystyle\sqrt{b\delta_++b^2+a}
    & \text{if \condA{} is not satisfied and}\; b\ge\dfrac{\alpha_--\delta_+}{2}\,,
    \\[1ex]
    \displaystyle\sqrt{b\alpha_-+a-\biggl(\dfrac{\alpha_--\delta_+}{2}\biggr)^2}
    & \text{if \condA{} is not satisfied and}\; b<\dfrac{\alpha_--\delta_+}{2}\,.
  \end{cases}
\]
\item
If $B$ is bounded, then one can choose $a=\|B\|^2$ and $b=0$, and hence
\begin{align}
  & \mu = \delta_++\|B\|, \qquad
  \xi_1 = \mu_+ = \alpha_--\|B\|,
  \label{muBbounded}\\[0.5ex]
  & \eta = \sqrt{\biggl(\|B\|^2
  -\biggl(\Bigl(\frac{\alpha_--\delta_+}{2}\Bigr)_+\biggr)^{\!2\,}\biggr)_+}\,.
  \label{etaBbounded}
\end{align}
\end{myenum}
\end{remark}

\bigskip

\noindent
Before we prove that $\BB$ contains $W^2(\cM)$ and $\sigma(\cM)$, we need some lemmas.

\begin{lemma}\label{le:nott}
Let $b\ge0$ and $a,t,\delta\in\RR$, and assume that
\begin{equation}\label{nott1}
  \biggl(\frac{t-\delta}{2}\biggr)^2 \le bt+a.
\end{equation}
Then
\begin{equation}\label{nott2}
  b\delta+b^2+a \ge 0 \qquad\text{and}\qquad
  \frac{t-\delta}{2} \le b+\sqrt{b\delta+b^2+a\,}.
\end{equation}
If strict inequality holds in \eqref{nott1}, then the inequalities
in \eqref{nott2} are also strict.
\end{lemma}

\begin{proof}
Relation \eqref{nott1} is equivalent to
\[
  t^2-2(\delta+2b)t+\delta^2-4a \le 0.
\]
The zeros of the polynomial in $t$ on the left-hand side are
\[
  t_\pm \defeq \delta+2b\pm\sqrt{(\delta+2b)^2-\delta^2+4a}
  = \delta+2b\pm2\sqrt{b\delta+b^2+a\,}.
\]
If \eqref{nott1} is satisfied, then the discriminant is non-negative
and $t_-\le t\le t_+$, which yields \eqref{nott2}.
If the inequality in \eqref{nott1} is strict, then $t_-<t<t_+$ and hence
the discriminant is strictly positive.
\end{proof}

\begin{lemma}\label{le:lapllb}
Assume that \condA{} is satisfied and
let $x\in\dom(\fra)\backslash\{0\}$ and $y\in\dom(\frd)\backslash\{0\}$.
Then $\lambda_\pm\binom{x}{y}\in\RR$.
Moreover, if \eqref{condA1} holds, then
\begin{equation}\label{lapllb1}
  \lambda_+\binom{x}{y} \ge -\frac{a}{b}\,;
\end{equation}
if \eqref{condA2} holds, then
\begin{equation}\label{lapllb2}
  \lambda_+\binom{x}{y} \ge \alpha_--b-\sqrt{\bigl(b\delta_++b^2+a\bigr)_+\!}\,.
\end{equation}
\end{lemma}

\begin{proof}
Let $\alpha$, $\beta$ and $\delta$ be as in \eqref{alphabetadelta}.
Suppose that $\lambda_\pm\binom{x}{y}\notin\RR$.
Then, by \eqref{lambdapmalpha} and \eqref{estbeta}, we have
\[
  \biggl(\frac{\alpha-\delta}{2}\biggr)^2 < |\beta|^2 \le b\alpha+a.
\]
This, together with Lemma~\ref{le:nott}, implies that
\[
  b\delta+b^2+a > 0 \qquad\text{and}\qquad
  \frac{\alpha-\delta}{2} < b+\sqrt{b\delta+b^2+a\,}.
\]
By the definition of $\alpha_-$ and $\delta_+$ we obtain
\[
  b\delta_++b^2+a > 0 \qquad\text{and}\qquad
  \frac{\alpha_--\delta_+}{2} < b+\sqrt{b\delta_++b^2+a\,},
\]
which is a contradiction to \condA.  Hence $\lambda_\pm\binom{x}{y}\in\RR$.

It follows again from \eqref{lambdapmalpha} and \eqref{estbeta} that
\begin{equation}\label{thlink}
  \lambda_+\binom{x}{y}
  \ge \frac{\alpha+\delta}{2} +\sqrt{\biggl(\frac{\alpha-\delta}{2}\biggr)^2-b\alpha-a\,}.
\end{equation}
Assume that \eqref{condA1} holds.  Then
\begin{equation}\label{897}
  b\delta+b^2+a \le 0
\end{equation}
and $b>0$.
Define the function
\[
  f(t) \defeq \frac{t+\delta}{2}+\sqrt{\biggl(\frac{t-\delta}{2}\biggr)^2-bt-a\,},
  \qquad t\in\RR,
\]
which is real-valued by \eqref{897}.
Its derivative is
\[
  f'(t) = \frac{1}{2}+\frac{\frac{t-\delta}{2}-b}{2\sqrt{\bigl(\frac{t-\delta}{2}\bigr)^2-bt-a}\,}
  = \frac{f(t)-(\delta+b)}{2\sqrt{\bigl(\frac{t-\delta}{2}\bigr)^2-bt-a}\,}\,,
\]
which implies that $f'(t)>0$ if and only if $f(t)>\delta+b$.
From this it follows that the sign of $f'$ is constant on $\RR$.
Since $f(t)\to\infty$ as $t\to\infty$, we obtain that $f$ is increasing on $\RR$
and $f(t)>\delta+b$ for all $t\in\RR$.
Relations \eqref{abalmi} and \eqref{897} imply that
\[
  \alpha \ge \alpha_- \ge -\frac{a}{b} \ge \delta+b \ge \delta.
\]
Hence, by \eqref{thlink},
\[
  \lambda_+\binom{x}{y} \ge f(\alpha) \ge f\Bigl(-\frac{a}{b}\Bigr)
  = -\frac{a}{b}\,,
\]
i.e.\ \eqref{lapllb1} holds.

Now assume that \eqref{condA2} is satisfied.  Note first that, for
$r,s\in\RR$ such that $r\ge0$ and $r\ge s$, one has
\[
  \sqrt{r-s} \ge \sqrt{r}-\sqrt{s_+}\,,
\]
which is easy to see.  From this and the relation $\frac{\alpha-\delta}{2}\ge b$
it follows that
\begin{align*}
  \lambda_+\binom{x}{y}
  &\ge \frac{\alpha+\delta}{2} + \sqrt{\biggl(\frac{\alpha-\delta}{2}\biggr)^2-b\alpha-a}
  \\[0.5ex]
  &= \frac{\alpha+\delta}{2} + \sqrt{\biggl(\frac{\alpha-\delta}{2}-b\biggr)^2-(b\delta+b^2+a)}
  \\[0.5ex]
  &\ge \frac{\alpha+\delta}{2} + \sqrt{\biggl(\frac{\alpha-\delta}{2}-b\biggr)^2}
  - \sqrt{(b\delta+b^2+a)_+}
  \\[0.5ex]
  &= \alpha-b-\sqrt{(b\delta+b^2+a)_+}
  \\[0.5ex]
  &\ge \alpha_--b-\sqrt{\bigl(b\delta_++b^2+a\bigr)_+\!}\,,
\end{align*}
i.e.\ \eqref{lapllb2} holds.
\end{proof}

\begin{lemma}\label{lamlem0}
Let $x\in\dom(\fra)\backslash\{0\}$ and $y\in\dom(\frd)\backslash\{0\}$ and let $\mu$
be as in \eqref{defmu}.  Then
\[
  \Re\lambda_-\binom{x}{y} \le \mu.
\]
\end{lemma}

\begin{proof}
Let $x\in\dom(\fra)\backslash\{0\}$ and $y\in\dom(\frd)\backslash\{0\}$ and let
$\alpha$, $\beta$ and $\delta$ be as in \eqref{alphabetadelta}.
Then
\[
  \lambda_- \defeq \lambda_-\binom{x}{y}
  = \frac{\alpha+\delta}{2}-\sqrt{\biggl(\frac{\alpha-\delta}{2}\biggr)^2-|\beta|^2}
\]
and \eqref{estbeta} is valid.

Let us first consider the case when
\[
  \biggl(\frac{t-\delta}{2}\biggr)^2 \le bt+a \qquad\text{for some }t\ge\alpha.
\]
It follows from Lemma~\ref{le:nott} that the inequalities in \eqref{nott2} hold, which
imply
\begin{align*}
  \Re\lambda_- &\le \frac{\alpha+\delta}{2} \le \frac{t+\delta}{2}
  \le \delta+b+\sqrt{b\delta+b^2+a\,}
  \\[0.5ex]
  &\le \delta_++b+\sqrt{b\delta_++b^2+a\,} = \mu.
\end{align*}

Now we consider the case when
\begin{equation}\label{case2}
  \biggl(\frac{t-\delta}{2}\biggr)^2 > bt+a \qquad\text{for all }t\ge\alpha.
\end{equation}
It follows from \eqref{case2} and \eqref{estbeta} that $\lambda_-\in\RR$ and
\begin{equation}\label{est1case2}
  \lambda_- \le \frac{\alpha+\delta}{2}-\sqrt{\biggl(\frac{\alpha-\delta}{2}\biggr)^2-b\alpha-a\,}.
\end{equation}
Define the function
\[
  f(t) \defeq \frac{t+\delta}{2}-\sqrt{\biggl(\frac{t-\delta}{2}\biggr)^2-bt-a}
\]
for such $t$ for which the expression under the square root is non-negative, i.e.\
either $\dom(f)=\RR$ or $\dom(f)=(-\infty,t_-]\cup[t_+,\infty)$ where $t_\pm$
are the zeros of the polynomial under the square root:
\[
  t_\pm = \delta+2b\pm2\sqrt{b\delta+b^2+a\,}.
\]
The derivative of $f$ is
\[
  f'(t) = \frac{1}{2}-\frac{\frac{t-\delta}{2}-b}{2\sqrt{\bigl(\frac{t-\delta}{2}\bigr)^2-bt-a\,}\,}
  = \frac{\delta+b-f(t)}{2\sqrt{\bigl(\frac{t-\delta}{2}\bigr)^2-bt-a\,}\,}\,,
\]
which implies that
\begin{equation}\label{fder_eqiv}
  f'(t) > 0 \qquad\Longleftrightarrow\qquad f(t) < \delta+b.
\end{equation}
If $a=b=0$, then $\beta=0$ and the assertion is clear since then $\lambda_-=\min\{\alpha,\delta\}$.
So assume that $a\ne0$ or $b\ne0$.  Then $f$ is not constant.
It follows from \eqref{fder_eqiv} that the sign of $f'$ is constant on each interval
in the domain of $f$.
Let us first consider the case when $\dom(f)=\RR$.
Since $f(t)\to-\infty$ as $t\to-\infty$, we have $f(t)<\delta+b$ for all $t\in\RR$
and hence (with \eqref{est1case2})
\[
  \lambda_- \le f(\alpha) < \delta+b \le \delta_++b \le \mu.
\]
Now consider the case when $\dom(f)\ne\RR$.  It follows from \eqref{case2}
that $\alpha\in[t_+,\infty)$.  Moreover,
\[
  f(t_+) = \frac{t_++\delta}{2} = \delta+b+\sqrt{b\delta+b^2+a\,} \ge \delta+b,
\]
which, by \eqref{fder_eqiv}, implies that $f'(t)\le0$ on $(t_+,\infty)$.
Hence (again with \eqref{est1case2})
\[
  \lambda_- \le f(\alpha) \le f(t_+) = \delta+b+\sqrt{b\delta+b^2+a\,}
  \le \delta_++b+\sqrt{b\delta_++b^2+a\,} = \mu,
\]
which proves the assertion also in this case.
\end{proof}

\noindent
The next proposition shows that the closure of the quadratic numerical range
is contained in $\BB$.

\begin{proposition}\label{numrange}
Suppose that Assumption~{\rm\ref{assump}} is satisfied.
Let $\cM$ be the operator as in Theorem~{\rm\ref{jsaness}},
$W^2(\cM)$ as in Definition~{\rm\ref{def:W2}} and $\BB$, $\mu$, $\mu_+$
as in Definition~{\rm\ref{def:Betc}}.
Then $\ov{W^2(\cM)}\subset\BB$.

Moreover, if \condA{} is satisfied, then $W^2(\cM)\subset\RR$ and
\begin{equation}\label{Lambdapmsep}
  \lambda_-\binom{x}{y} \le \mu, \quad \lambda_+\binom{x}{y} \ge \mu_+
  \qquad \text{for}\;\; x\in\dom(\fra)\backslash\{0\},\;y\in\dom(\frd)\backslash\{0\}.
\end{equation}
\end{proposition}

\begin{proof}
Since $\BB$ is closed, it suffices to prove that $W^2(\cM)\subset\BB$.
Let $z\in W^2(\cM)$.
Then there exist $x\in\dom(\fra)\backslash\{0\}$ and $y\in\dom(\frd)\backslash\{0\}$
such that $z=\lambda_+\binom{x}{y}$ or $z=\lambda_-\binom{x}{y}$.
Let $\alpha$, $\beta$ and $\delta$ be as in \eqref{alphabetadelta}.

First assume that $z\in\RR$.  If condition \condA{} is satisfied, then,
by Lemmas~\ref{le:lapllb} and \ref{lamlem0}, we have
either $z=\lambda_+\binom{x}{y}\ge\mu_+$ or $z=\lambda_-\binom{x}{y}\le\mu$,
which also shows \eqref{Lambdapmsep}.
If $D$ is bounded, then
\[
  z \ge \lambda_-\binom{x}{y} \ge \frac{\alpha+\delta}{2}-\biggl|\frac{\alpha-\delta}{2}\biggr|
  = \min\{\alpha,\delta\} \ge \min\{\alpha_-,\delta_-\}
  = \mu_-,
\]
which shows that $z\in\BB$ when $z\in\RR$.

Now assume that $z\notin\RR$.
Using \eqref{estbeta} and the relation $t^2\ge ((t)_+)^2$ for $t\in\RR$
we obtain for the imaginary part of $z$ that
\begin{align*}
  |\Im z| &= \sqrt{\biggl(|\beta|^2-\biggl(\frac{\alpha-\delta}{2}\biggr)^2\biggr)_+}
  \le \sqrt{\biggl(b\alpha+a-\biggl(\frac{\alpha-\delta}{2}\biggr)^2\biggr)_+}
  \\[0.5ex]
  &= \sqrt{\biggl(b\delta+b^2+a-\biggl(\frac{\alpha-\delta}{2}-b\biggr)^2\biggr)_+}
  \displaybreak[0]\\[0.5ex]
  &\le \sqrt{\biggl(b\delta+b^2+a-\biggl(\Bigl(\frac{\alpha-\delta}{2}-b\Bigr)_+\biggr)^2\biggr)_+}
  \\[0.5ex]
  &\le \sqrt{\biggl(b\delta_++b^2+a-\biggl(\Bigl(\frac{\alpha_--\delta_+}{2}-b\Bigr)_+\biggr)^2\biggr)_+}\,.
\end{align*}
The upper bound for $\Re z$ follows directly from Lemma~\ref{lamlem0}.
For the lower bound observe that
\[
  0 > \biggl(\frac{\alpha-\delta}{2}\biggr)^2 - |\beta|^2
  \ge \biggl(\frac{\alpha-\delta}{2}\biggr)^2 - b\alpha - a.
\]
Hence $b\alpha+a>0$ and
\[
  \frac{\alpha-\delta}{2} < \sqrt{b\alpha+a\,},
\]
which implies that
\begin{equation}\label{499}
  \Re z = \frac{\alpha+\delta}{2} > \alpha-\sqrt{b\alpha+a\,}.
\end{equation}
If $b=0$, then the right-hand side of \eqref{499} is bounded from
below by $\alpha_--\sqrt{a}$, which is equal to $\xi_1$ in that case.
For the case $b>0$ we consider the function
\[
  f(t) \defeq t-\sqrt{bt+a\,}, \qquad t\in\Bigl[-\frac{a}{b},\infty\Bigr),
\]
which attains its minimum at $t_0\defeq \frac{b}{4}-\frac{a}{b}$\,.
If $t_0\le\alpha_-$, then
\[
  \min_{t\in[\alpha_-,\infty)}f(t) = f(\alpha_-) = \alpha_--\sqrt{b\alpha_-+a}\,.
\]
If $t_0>\alpha_-$, then
\[
  \min_{t\in[\alpha_-,\infty)}f(t) = f(t_0) = t_0-\frac{b}{2} > \alpha_--\frac{b}{2}\,.
\]
Hence $\Re z\ge \xi_1$ also in this case.

If $D$ is bounded, then one also has $\delta\ge\delta_-$ and hence
\[
  \Re z = \frac{\alpha+\delta}{2} \ge \frac{\alpha_-+\delta_-}{2} = \xi_2.
\]
This shows that $\Re z\ge\xi_-$ in all cases and hence $z\in\BB$.
\end{proof}

Next we need an auxiliary lemma before we prove the spectral inclusion.
For a similar result for certain diagonally dominant block operator matrices
we refer to \cite[Theorem~4.2]{T09}.

\begin{lemma}\label{appB}
Suppose that Assumption~{\rm\ref{assump}} is satisfied and
let $z\in\CC\backslash(\delta_+,\delta_++b_0)$.
Then $z\notin\ov{W^2(\cM)}$ implies that $\cM-z$ has closed range.
\end{lemma}

\begin{proof}
We show the contraposition.  Let $z\in\CC\setminus(\delta_+,\delta_++b_0)$
and suppose that $\ran(\cM-z)$ is not closed.
Then, $z\in\sigmaapp(\cM)$, i.e.\ there exists a sequence $(x_n,y_n)^T\in\dom(\cM)$ with
\[
  (\cM-z)\binom{x_n}{y_n} \to 0 \qquad\text{and}\qquad
  \|x_n\|^2 + \|y_n\|^2 = 1 \quad\forall\, n\in\mathbb{N};
\]
see \cite[Theorem~IV.5.2]{katopert}.
We have to show that $z\in\ov{W^2(\cM)}$.

If $\dim\cH_1=1$ or $\dim\cH_2=1$, then $B$ is bounded,
and hence \cite[Corollary~4.3]{T09} implies that $z\in\ov{W^2(\cM)}$.
If $A$ is bounded, then $B$ is bounded, and again $z\in\ov{W^2(\cM)}$.
For the rest of the proof assume that $\dim\cH_1\ge2$, $\dim\cH_2\ge2$
and that $A$ is unbounded.

It follows from \eqref{form1} and \eqref{form2} that
\begin{align}
  \fra[x_n]-z\|x_n\|^2 + \langle y_n,B^*x_n\rangle \to 0,
  \label{ul1} \\[0.5ex]
  -\langle B^*x_n,y_n\rangle + \frd[y_n] -z\|y_n\|^2 \to 0.
  \label{ul2}
\end{align}

First we consider the case when $z\in\CC\backslash\RR$.
Taking the imaginary parts of the left and the right-hand sides of \eqref{ul1} and \eqref{ul2}
we obtain
\[
  -\Im z\|x_n\|^2 + \Im\langle y_n,B^*x_n\rangle \to 0 \quad\text{and}\quad
  -\Im\langle B^*x_n,y_n\rangle-\Im z\|y_n\|^2 \to 0.
\]
If we take the difference and observe that $\Im z\ne0$, we get $\|x_n\|-\|y_n\|\to0$ and thus
\begin{equation}\label{xy}
  \|x_n\|\to\frac{1}{\sqrt{2}\,} \qquad\text{and}\qquad \|y_n\|\to\frac{1}{\sqrt{2}\,}\,.
\end{equation}
Lemma~\ref{le:approx_es}\,(i) implies that $\fra[x_n]$ and $\langle B^*x_n,y_n\rangle$
are bounded.  By \eqref{ul2} also $\frd[y_n]$ is bounded.
From \eqref{ul1}, \eqref{ul2} and \eqref{xy} it follows that
\begin{equation}\label{mcon}
  \bigl(\cM_{x_n,y_n} - z\bigr)\begin{pmatrix} \|x_n\| \\[1ex] \|y_n\| \end{pmatrix}
  = \begin{pmatrix} \dfrac{\fra[x_n]-z\|x_n\|^2+\langle y_n,B^*x_n\rangle}{\|x_n\|} \\[3ex]
  \dfrac{-\langle B^*x_n,y_n\rangle+\frd[y_n]-z\|y_n\|^2}{\|y_n\|} \end{pmatrix}
  \to 0.
\end{equation}
Since all entries of $\cM_{x_n,y_n}$ are bounded, \eqref{mcon} and \eqref{xy}
imply that
\[
  \det(\cM_{x_n,y_n}-z) \to 0.
\]
Hence there exists a sequence $z_n\in\sigma(\cM_{x_n,y_n})\subset W^2(\cM)$
such that $z_n\to z$, which shows that $z\in\ov{W^2(\cM)}$.

Now let $z\in\RR$.  Taking the sum of the real parts of the left-hand sides
of \eqref{ul1} and \eqref{ul2} we obtain
\[
  \fra[x_n]-z\|x_n\|^2+\frd[y_n]-z\|y_n\|^2 \to 0.
\]
If $z<\mu_-$, i.e.\ $D$ is bounded and $z<\alpha_-$ and $z<\delta_-$, then
\begin{align*}
  &\fra[x_n]-z\|x_n\|^2+\frd[y_n]-z\|y_n\|^2
  \le (\alpha_--z)\|x_n\|^2 + (\delta_--z)\|y_n\|^2
  \\[0.5ex]
  &\le \max\bigl\{\alpha_--z,\delta_--z\bigr\}\bigl(\|x_n\|^2+\|y_n\|^2\bigr)
  =  \max\bigl\{\alpha_--z,\delta_--z\bigr\} < 0,
\end{align*}
which is a contradiction.

If $\delta_-\le z\le\delta_+$, then $z\in\ov{W(D)}$;
if $z\ge\alpha_-$, then $z\in\ov{W(A)}$ since we assumed that $A$ is unbounded.
In both cases it follows from Lemma~\ref{le:qnrnr} that $z\in\ov{W^2(\cM)}$.

Finally, assume that $z\in(\delta_++b_0,\alpha_-)$.
Since $z\in\UU$ in this case, we have $\liminf_{n\to\infty}\|x_n\|>0$
by Lemma~\ref{le:approx_es}\,(iii).
If $y_{n_k}\to0$ for a subsequence $y_{n_k}$, then \eqref{ul1} implies
that $\fra[x_{n_k}]-z\|x_{n_k}\|^2\to0$, which is
a contradiction to the fact that $\|x_{n_k}\|\to1$ and $z<\alpha_-$.
Hence also $\liminf_{n\to\infty}\|y_n\|>0$ and we can argue as in the
case $z\in\CC\backslash\RR$ to obtain that $z\in\ov{W^2(\cM)}$.
\end{proof}

The next proposition shows that, essentially, the spectrum of $\cM$ is
contained in the closure of the quadratic numerical range.
Only in the interval $(\delta_+,\delta_++b_0)$ we are not able to prove
such a spectral inclusion.
For other types of block operator matrices results about spectral inclusion were
shown in many papers; see, e.g.\ \cite[Theorem~2.1]{LT98}, \cite[Theorem~2.3]{LMMT01}
and \cite[Theorem~4.2]{T09}.

\begin{proposition}\label{pr:spec_incl}
Suppose that Assumption~{\rm\ref{assump}} is satisfied and
let $\cM$ be the operator as in Theorem~{\rm\ref{jsaness}}.
Moreover, let $z\in\CC\backslash(\delta_+,\delta_++b_0)$.
Then $z\in\sigma(\cM)$ implies that $z\in\ov{W^2(\cM)}$.
\end{proposition}

\begin{proof}
Assume that $z\notin\ov{W^2(\cM)}$.
It follows from Lemma~\ref{appB} that $\ran(\cM-z)$ is closed.
Moreover, Lemma~\ref{qnrps} applied to $\cM$ and $\cM^*$
yields $\nul(\cM- z)=0$ and $\nul(\cM^*-\ov{z})=0$.
The latter implies that $\defect(\cM-z)=0$;
see, e.g.\ \cite[Theorem~IV.5.13]{katopert}.  Hence $z\in\rho(\cM)$.
\end{proof}

The next theorem shows that the spectrum of $\cM$ is contained in $\BB$.

\begin{theorem}\label{th:specinB}
Suppose that Assumption~{\rm\ref{assump}} is satisfied, let $\cM$ be the operator
as in Theorem~{\rm\ref{jsaness}} and let $\BB$ as in \eqref{defB}.
Then $\sigma(\cM)\subset\BB$.
In particular, if condition \condA{} is satisfied, then $\sigma(\cM)\subset\mathbb{R}$.
\end{theorem}

\begin{proof}
Let $z\in\sigma(\cM)$.
If $z\in\CC\backslash(\delta_+,\delta_++b_0)$,
then $z\in\ov{W^2(\cM)}\subset\BB$ by Propositions~\ref{pr:spec_incl} and \ref{numrange}.
If $z\in(\delta_+,\delta_++b_0)$, then $z\in\BB$ since
$\mu_- \le \delta_+$ and $\delta_++b_0 \le \delta_++b \le \mu$.
\end{proof}

When $B$ is a bounded operator, then $\eta$, which bounds the imaginary parts of
spectral points, is given by \eqref{etaBbounded};
this was proved in \cite[Theorem~5.5\,(iii)]{T09}.

The above theorem shows that the spectrum is real provided the spectra of the
diagonal components are sufficiently separated and $B$ is not ``too large''.
As the following result shows,
this can be particularly straightforward when $B$ is bounded;
see also \cite[Proposition~2.6.8]{tretbook} and \cite[Theorem~5.5]{T09}.

In the next corollary, which follows immediately from Theorem~\ref{th:specinB}
and Remark~\ref{re:pos_gap}\,(viii), we consider the situation when $B$ is bounded.
The estimate for the imaginary part in \eqref{Bbdd_specincl2} was also proved
in \cite[Theorem~5.5]{T09}.
A slightly better enclosure for $\sigma(\cM)$ than \eqref{Bbdd_specincl1}
was obtained in \cite[Theorem~5.8]{AMS09} and \cite[Theorem~5.4]{AMT10}.

\begin{corollary}\label{co:boundedB}
Suppose that Assumption~{\rm\ref{assump}} is satisfied and that $B$ is bounded.
If
\[
  \|B\| \le \frac{\alpha_--\delta_+}{2}\,,
\]
then
\begin{equation}\label{Bbdd_specincl1}
  \sigma(\cM) \subset \bigl(-\infty,\delta_++\|B\|\,\bigr]\cup\bigl[\alpha_--\|B\|,\infty\bigr).
\end{equation}
Otherwise,
\begin{equation}\label{Bbdd_specincl2}
\begin{aligned}
  \sigma(\cM) \subset \RR\cup\Biggl\{z\in\CC\backslash\RR:\;\;
  & \alpha_--\|B\| \le \Re z \le \delta_++\|B\|,
  \\[1ex]
  & |\Im z| \le \sqrt{\|B\|^2-\Bigl(\Bigl(\frac{\alpha_--\delta_+}{2}\Bigr)_+\Bigr)^2}
  \,\Biggr\}.
\end{aligned}
\end{equation}
If $D$ is bounded with $\delta_-=\min\sigma(D)$, then $(-\infty,\delta_-)\subset\rho(\cM)$
and $\Re z\ge\frac{\alpha_-+\delta_-}{2}$ for $z\in\sigma(\cM)\backslash\RR$.
\end{corollary}

\begin{proof}
Since $B$ is bounded, we can chose $a=\|B\|^2$ and $b=0$.  Under our assumptions
the inequality \eqref{condA2} is satisfied.
Hence \eqref{Bbdd_specincl1} holds by Theorem~\ref{th:specinB} and
the definition of $\BB$.
\end{proof}

\begin{remark}
Even if $B$ is bounded, it may be possible to choose $a$ and $b$ such that $b>0$
to obtain better enclosures for the spectrum, in particular if $\mu_+=\alpha_-$
with such a choice; see Remark~\ref{re:pos_gap}\,(iv).
\end{remark}

\begin{remark}
Let us consider the family of operators
\[
  \cM_t \defeq \begin{pmatrix} A & tB \\[0.5ex] -tB^* & D \end{pmatrix},
  \qquad t\in[0,\infty),
\]
which was also studied in \cite{LLT02}.
Clearly, if Assumption~\ref{assump} is satisfied for $t=1$, then it is
satisfied for all $t\in[0,\infty)$.
If $\delta_+<\alpha_-$, i.e.\ the spectra of $A$ and $D$ are separated,
then there exists a $t_0>0$ such that, for $t\in[0,t_0]$, condition~\condA{}
is satisfied and hence $\sigma(\cM_t)\subset\RR$.
If $\delta_+\ge\alpha_-$, it may happen that the spectrum of $\cM$ is non-real
for every positive $t$.

If $\delta_+<\alpha_-$, then, in general, the gap $(\delta_+,\alpha_-)$ in the spectrum
closes from both endpoints with increasing $t$.
However, if, e.g.\ $\alpha_-=0$ and $a=0$, $b>0$ in \eqref{domcons},
then $\mu_+=\alpha_-$ as long as \eqref{condA1} is satisfied, i.e.\
the gap closes only from the left endpoint.

If $D$ is bounded and $\delta_-=\min\sigma(D)$, then for all $t\in[0,\infty)$,
the set $\sigma(\cM_t)\cap\RR$
is bounded from below by $\min\{\alpha_-,\delta_-\}$
and the real parts of points from $\sigma(\cM_t)\backslash\RR$ are bounded
from below by $\frac{\alpha_-+\delta_-}{2}$\,.
\end{remark}

In the next section we characterise elements from $\sigma(\cM)$ in the interval $(\mu,\infty)$
with variational principles.  Since the proof uses the Schur complement, we must
ensure that $S$ and $\cM$ have the same essential spectrum in $(\mu,\infty)$.
Note that $(\mu,\infty)\subset\UU$ and hence $S(\lambda)$ is well defined
for $\lambda\in(\mu,\infty)$.

\begin{theorem}\label{th:sess_equiv}
Suppose that Assumption~{\rm\ref{assump}} is satisfied, let $\cM$ be the operator
as in Theorem~{\rm\ref{jsaness}}, let $S$ be its Schur complement and
let $\mu$ be as in \eqref{defmu}.  Then
\begin{equation}\label{essspeq}
  \sigmaess(S)\cap(\mu,\infty) = \sigmaess(\cM)\cap(\mu,\infty).
\end{equation}
\end{theorem}

\begin{proof}
Let $z\in\sigmaess(S)\cap(\mu,\infty)$.
Since $0\in\sigmaess(S(z))$ and $S(z)$ is self-adjoint, the operator $S(z)$ is
not semi-Fredholm with $\nul(S(z))<\infty$.  By \cite[Theorem~IX.1.3]{EdmundsEvans}
there exists a singular sequence for $S(z)$ corresponding to $0$,
i.e.\ there exist $x_n\in\dom(S(z))$, $n\in\NN$, such that
\[
  \|x_n\|=1, \qquad S(z)x_n \to 0, \qquad x_n \rightharpoonup 0.
\]
Set
\[
  y_n \defeq (D-z)^{-1}B^*x_n, \qquad
  w_n \defeq \binom{x_n}{y_n}, \qquad
  \qquad\text{and}\qquad
  \hat w_n\defeq \frac{w_n}{\|w_n\|}\,.
\]
From Lemma~\ref{le:Sz_T_z}\,(i) we obtain that $\hat w_n\in\dom(\cM)$ and
\[
  (\cM-z)\hat w_n = \frac{1}{\|w_n\|}\begin{pmatrix} S(z)x_n \\[0.5ex] 0 \end{pmatrix}
  \to 0;
\]
note that $\|w_n\|\ge1$.
Moreover, for $u$ in the dense set $\dom(B(D-z)^{-1})$ we have
\[
  \langle y_n,u\rangle = \bigl\langle x_n,B(D-z)^{-1}u\bigr\rangle \to 0.
\]
Since $y_n$ is bounded by Lemma~\ref{le:approx_es}\,(i), we have $y_n\rightharpoonup0$
and therefore $\hat w_n\rightharpoonup0$.
Hence $\hat w_n$ is a singular sequence for $\cM$ corresponding to $z$.
Again from \cite[Theorem~IX.1.3]{EdmundsEvans} we obtain that $z\in\sigmaess(\cM)$.
This shows the inclusion ``$\subset$'' in \eqref{essspeq}.

Now let $z\in\sigmaess(\cM)\cap(\mu,\infty)$ and suppose that $z\notin\sigmaess(S)$.
It follows from Theorem~\ref{specequiv} that $z\in\sigma(S)$ and that
\begin{equation}\label{nulnon}
  0 < \nul(S(z)) = \nul(\cM-z) < \infty.
\end{equation}
Since $S(z)$ is self-adjoint, we also have $0\in\sigmadis(S(z))$.
Suppose that $\cM-z$ has closed range.  Then $\cM-z$ is semi-Fredholm with $\defect(\cM-z)=\infty$.
Let $z_n=z+i/n$, $n\in\NN$.  Then $z_n\to z$, $z_n\notin\BB$ and
hence $z_n\in\rho(\cM)$ by Theorem~\ref{th:specinB}.
It follows from \cite[Theorem~IV.5.17]{katopert} that
$\ind(\cM-z_n)=\ind(\cM-z)$ for large enough $n$, which is a contradiction
since $\ind(\cM-z_n)=0$ for all $n\in\mathbb{N}$ and $\ind(\cM-z)=-\infty$.
Hence $\ran(\cM-z)$ is not closed.
Therefore, by \cite[Theorem~IV.5.2]{katopert}, there exists a sequence
of vectors $(x_n,y_n)^T\in\dom(\cM)$ with
\begin{equation}\label{perpvec}
  \binom{x_n}{y_n} \perp \ker(\cM-z) \quad\text{and}\quad
  \|x_n\|^2+\|y_n\|^2 = 1 \quad\text{for each $n\in\NN$}
\end{equation}
such that
\[
  \binom{u_n}{v_n} \defeq
  (\cM-z)\binom{x_n}{y_n}
  = \begin{pmatrix}
    (A-\nu)\bigl(x_n + \overline{(A-\nu)^{-1}B}y_n\bigr) + (\nu-z)x_n \\[1ex]
    -B^*x_n + (D-z)y_n
  \end{pmatrix} \to 0.
\]
Let $P$ be the orthogonal projection from $\mathcal{H}_1$ onto $\ker(S(z))$,
set $\tilde{x}_n=(I-P)x_n$ and let $\tilde{S}(z)$ be the restriction of $S(z)$
to the Hilbert space $(I-P)\mathcal{H}_1$, which has a bounded inverse
since $0\in\sigmadis(S(z))$.
Set $\xi_n\defeq\tilde{S}(z)^{-1}\tilde x_n$.  Since $B^*\tilde{S}(z)^{-1}$ is a bounded
operator by the closed graph theorem, the assumptions on $\xi_n$
in Lemma~\ref{le:approx_es}\,(iii) are satisfied.
The latter implies that
\[
  \|\tilde x_n\|^2 = \frs(z)\bigl[x_n,\tilde{S}(z)^{-1}\tilde x_n\bigr] \to 0.
\]
Since $\ker(S(z))$ is finite-dimensional, there exists a subsequence $x_{n_k}$
such that $x_{n_k}\to x\in\ker(S(z))$.
It follows from Lemma~\ref{le:approx_es}\,(iv) that $x\in\dom(\fra)$ and $B^*x_{n_k}\to B^*x$.
Hence
\[
  \binom{x_{n_k}}{y_{n_k}} \to \begin{pmatrix} x \\[0.5ex] (D-z)^{-1}B^*x \end{pmatrix}
  \in \ker(\cM-z)
\]
by Lemma~\ref{le:Sz_T_z}.
As this contradicts \eqref{perpvec}, we have $z\in\sigmaess(S)$.
Hence the reverse inclusion in \eqref{essspeq} is also shown.
\end{proof}

\begin{corollary}\label{compres}
If $A$ has compact resolvent, then $(\mu,\infty)\cap\sigmaess(\cM)=\varnothing$.
\end{corollary}

\begin{proof}
In view of Theorem~\ref{th:sess_equiv}, it is sufficient to show
that $(\mu,\infty)\cap\sigmaess(S)=\varnothing$.
Let $z\in(\mu,\infty)$ and $x\in\dom(\fra)$.  It follows from \eqref{est_DB} that
\[
  \bigl|\bigl\langle(D-z)^{-1}B^*x,B^*x\bigr\rangle\bigr|
  \le \frac{b}{z-\delta_+}\fra[x]+\frac{a}{z-\delta_+}\|x\|^2,
\]
Since $z-\delta_+>b$, this, together with \cite[Theorem~VI.3.4]{katopert}, implies
that $S(z)$ has compact resolvent.
\end{proof}

Recall that under the extra assumption~\ref{assump}.(I) more can be said
about $\sigmaess(\cM)$; see Proposition~\ref{pr:sessALMS}.

\section{Variational Principles}\label{varprince}

\noindent
In this section we prove variational principles that characterise eigenvalues of the
operator $\cM$ and the Schur complement $S$ in a certain interval.
The functionals in these variational principles are connected either with the Schur complement
or the quadratic numerical range of the operator $\cM$.

First we recall a property of operator functions that was used in
\cite[Lemma~2]{VM74} by A.~Virozub and V.~Matsaev for functions whose
values are bounded operators; see also, e.g.\ \cite{LLMT08,LMM05}.
In \cite{LS} this property was introduced for certain functions whose values
are unbounded operators.
Here we formulate it for families of quadratic forms and apply it then to
holomorphic operator functions of type~(B).

\begin{definition}
Let $\Delta\subset\RR$ be an interval and let $\frt(\lambda)$, $\lambda\in\Delta$,
be a family of
closed symmetric quadratic forms such that $\dom(\frt(\lambda))$ is independent of $\lambda$
and such that $\frt(\cdot)[x]$ is differentiable for each $x\in\dom(\frt(\lambda))$.
We say that $\frt(\cdot)$ satisfies the condition \VM{} on the interval $\Delta$ if,
for each compact subinterval $I\subset\Delta$,
there exist $\eps,\delta>0$ such that, for all $\lambda\in I$ and
all $x\in\dom(\frt(\lambda))$,
\begin{equation}\label{VMimpl}
  \bigl|\frt(\lambda)[x]\bigr| \le \eps\|x\|^2
  \quad\Longrightarrow\quad
  \frt'(\lambda)[x]\le-\delta\|x\|^2.
\end{equation}
\end{definition}

\medskip

The condition implies in particular that if $\lambda_0$ is an inner point
of $\Delta$ and $|\frt(\lambda_0)[x]|$ is small enough, then
$\frt(\cdot)[x]$ must have a zero close to $\lambda_0$.

\begin{lemma}\label{vmlem}
Let $\frs(\lambda)$, $\lambda\in\UU$, be the quadratic forms from
Definition~{\rm\ref{def:schur}} associated with the Schur complement
of the operator $\cM$ and let $\mu$ be as in \eqref{defmu}.
Then $\frs$ satisfies the condition \VM{} on the interval $(\mu,\infty)$.
\end{lemma}

\begin{proof}
First note that
\begin{align}
  \frs(\lambda)[x] &= \fra[x]-\lambda\|x\|^2-\bigl\|(\lambda-D)^{-\frac12}B^*x\bigr\|^2,
  \label{frs1} \\[0.5ex]
  \frs'(\lambda)[x] &= -\|x\|^2+\bigl\|(\lambda-D)^{-1}B^*x\bigr\|^2.
  \label{frs2}
\end{align}
Let $\eps>0$ be arbitrary for the moment, let $\lambda\in(\mu,\infty)$
and let $x\in\dom(\fra)=\dom(\frs(\lambda))$ such that $|\frs(\lambda)[x]|\le\eps\|x\|^2$.
It follows from \eqref{domcons} and \eqref{frs1} that
\begin{align*}
  \bigl\|(\lambda-D)^{-\frac12}B^*x\bigr\|^2
  &\le \frac{\|B^*x\|^2}{\lambda-\delta_+}
  \le \frac{b\fra[x]+a\|x\|^2}{\lambda-\delta_+} \\[0.5ex]
  &= \frac{b\bigl(\frs(\lambda)[x]+\lambda\|x\|^2+\bigl\|(\lambda-D)^{-\frac12}B^*x\bigr\|^2\bigr)
  +a\|x\|^2}{\lambda-\delta_+}\,.
\end{align*}
Rearranging this inequality we obtain
\[
  (\lambda-\delta_+-b)\bigl\|(\lambda-D)^{-\frac12}B^*x\bigr\|^2
  \le b\frs(\lambda)[x]+(b\lambda+a)\|x\|^2
  \le (b\eps+b\lambda+a)\|x\|^2.
\]
Since $\lambda>\delta_++b$, we have
\begin{align*}
  \frs'(\lambda)[x]
  &\le -\|x\|^2+\bigl\|(\lambda-D)^{-\frac12}\bigr\|^2\bigl\|(\lambda-D)^{-\frac12}B^*x\bigr\|^2
  \\[0.5ex]
  &\le \biggl(-1+\frac{1}{\lambda-\delta_+}\cdot\frac{b\eps+b\lambda+a}{\lambda-\delta_+-b}\biggr)\|x\|^2
  = \bigl(g(\lambda)+h(\lambda)\eps\bigr)\|x\|^2,
\end{align*}
where
\[
  g(\lambda) \defeq -1+\frac{b\lambda+a}{(\lambda-\delta_+)(\lambda-\delta_+-b)}\,, \qquad
  h(\lambda) \defeq \frac{b}{(\lambda-\delta_+)(\lambda-\delta_+-b)}\,.
\]
Moreover, $g(\lambda)<0$ if and only if
$\lambda^2-2(\delta_++b)\lambda+\delta_+^2+b\delta_+-a>0$;
it is easily seen that the latter inequality is true for $\lambda\in(\mu,\infty)$.
Now let $I$ be a compact subinterval of $(\mu,\infty)$.  Since $g$ is continuous
on $(\mu,\infty)$ and $I$ is compact, there exists a $c<0$
such that $g(\lambda)\le c$ for $\lambda\in I$.
Choose $\eps>0$ so small that $\eps h(\lambda)\le c/2$ for $\lambda\in I$.
Then, with $\delta\defeq c/2$, we have $\frs'(\lambda)[x]\le -\delta\|x\|^2$
for $\lambda\in I$.
\end{proof}

The previous lemma implies that if the function $\frs(\cdot)[x]$, for $x\in\dom(\fra)\backslash\{0\}$,
has a zero, then the derivative is negative at this zero.
In particular, for each $x\in\dom(\fra)\backslash\{0\}$ the function $\frs(\cdot)[x]$
is decreasing at value zero
(in the terminology of \cite{BEL00} and \cite{EL04}) and
hence has at most one zero in $(\mu,\infty)$.
Moreover, $\frs(\lambda)[x]\to-\infty$ as $\lambda\to\infty$.

Next we define a generalised Rayleigh functional, which is used in the variational
principle below.  This functional generalises the Rayleigh quotient for linear operators
to the situation of an operator function;
for more general operator functions it has been defined in \cite{BEL00}
and \cite{EL04}.

\begin{definition}\label{def:rayleigh}
We define the \emph{generalised Rayleigh functional}
$p:\dom(\fra)\backslash\{0\}\to\RR\cup\{-\infty\}$ as follows
\begin{alignat*}{2}
  p(x) &= \lambda_0 & &\text{if } \frs(\lambda_0)[x]=0\text{ for a }\lambda_0\in(\mu,\infty), \\[1ex]
  p(x) &= -\infty  \qquad &
  &\text{if }\frs(\lambda)[x]<0 \;\;\text{for all } \lambda\in(\mu,\infty).
\end{alignat*}
\end{definition}

\begin{remark}\label{rem:rayleighinf}
In the case when $\frs(\lambda)[x]<0$ for all $\lambda\in(\mu,\infty)$ one can also
set $p(x)$ equal to any number in $(-\infty,\mu]$ (which may depend on $x$);
see \cite[\S2]{JLT}.
\end{remark}

\medskip

\noindent
Before we formulate the next theorem we introduce another notation that is needed.

\begin{definition}\label{def:kappami}
For a self-adjoint operator $T$ denote by $\kappa_-(T)$ the dimension of the
spectral subspace for $T$ corresponding to the interval $(-\infty,0)$.
\end{definition}

The next theorem contains a variational principle for eigenvalues of $\cM$
in the interval $(\mu,\infty)$.  Note that there is a shift in the index: in general,
the index of the eigenvalue does not match the dimension of the corresponding subspace
in the variation.
For bounded $A$, $B$ and $D$ a similar but slightly weaker result was proved
in \cite[\S4.3]{BEL00}.

\begin{theorem}\label{finitevar}
Suppose that Assumption~{\rm\ref{assump}} is satisfied, let $\cM$ be the operator
as in Theorem~{\rm\ref{jsaness}}, let $\mu$ be as in \eqref{defmu}
and let $p$ and $\kappa_-$ be as in Definitions~{\rm\ref{def:rayleigh}}
and {\rm\ref{def:kappami}}.
Assume that
\begin{equation}\label{condkappa}
  \exists\,\gamma\in(\mu,\infty) \quad\text{such that}\quad
  \kappa_-(S(\gamma))<\infty.
\end{equation}
Then
\begin{equation}\label{deflae}
  \mu < \lae \defeq \begin{cases}
    \inf\bigl(\sigmaess(\cM)\cap(\mu,\infty)\bigr)
    & \text{if } \sigmaess(\cM)\cap(\mu,\infty)\ne\varnothing, \\[1ex]
    \infty & \text{otherwise}.
  \end{cases}
\end{equation}
Moreover, $\sigma(\cM)\cap(\mu,\lae)$ is at most countable, consists of eigenvalues only
and has $\lae$ as only possible accumulation point.

Let $\gamma_0\in(\mu,\lae)$ be arbitrary and
let $(\lambda_j)_{j=1}^N$, $N\in\mathbb{N}_0\cup\{\infty\}$, be the
finite or infinite sequence of eigenvalues of\, $\cM$ in the interval $[\gamma_0,\lae)$
in non-decreasing order and repeated according to multiplicities.  Then
\begin{equation}\label{defkappa}
  \kappa \defeq \kappa_-(S(\gamma_0)) < \infty
\end{equation}
and
\begin{equation}\label{var1}
  \lambda_n = \min_{\substack{\cL\subset\dom(\fra) \\[0.2ex] \dim\cL=\kappa+n}}\;
  \max_{x\in\cL\setminus\{0\}}\;\;p(x)
  = \max_{\substack{\cL\subset\mathcal{H} \\[0.2ex] \dim\cL=\kappa+n-1}}
  \inf_{\substack{x\in\dom(\fra)\setminus\{0\} \\[0.2ex] \perp\cL}}\;\;p(x)
\end{equation}
for $n\in\NN$, $n\le N$.
Moreover, if\, $N$ is finite and $\cH_1$ is infinite-dimensional,
then $\lae<\infty$ and
\begin{equation}\label{var2}
  \lae = \inf_{\substack{\cL\subset\dom(\fra) \\[0.2ex] \dim\cL=\kappa+n}}\;
  \max_{x\in\cL\setminus\{0\}}\;\;p(x)
  = \sup_{\substack{\cL\subset\mathcal{H} \\[0.2ex] \dim\cL=\kappa+n-1}}
  \inf_{\substack{x\in\dom(\fra)\setminus\{0\} \\[0.2ex] x\perp\cL}}\;\;p(x)
\end{equation}
for $n>N$.
\end{theorem}

\begin{proof}
Note first that $\sigma_*(\cM)\cap(\mu,\infty)=\sigma_*(S)\cap(\mu,\infty)$
for $\sigma_*=\sigma,\,\sigmap,\,\sigmaess$; see Theorems~\ref{specequiv}
and \ref{th:sess_equiv}.
Lemma~\ref{forms} implies that the Schur complement $S$
is a holomorphic operator family of type (B).
Moreover, $S$ satisfies the condition \VM{} on $(\mu,\infty)$ by Lemma~\ref{vmlem}.
Hence the assumptions (i)--(v) in \cite[\S2]{EL04} are satisfied.

Suppose that $\lae=\mu$.  Then, for every $\eps>0$, there exists a $\lambda\in(\mu,\mu+\eps)$
such that $\lambda\in\sigmaess(S)$.  This, together with \cite[Lemma~2.9]{EL04},
implies that $\kappa_-(S(t))=\infty$ for all $t>\lambda$, a contradiction.
Hence $\lae>\mu$.
Now almost all remaining assertions follow immediately from \cite[Theorem~2.1]{EL04}.
We only have to show that $\lae<\infty$ if $N<\infty$ and $\dim\cH_1=\infty$.
Suppose that $\lae=\infty$.  Then, by \cite[Theorem~2.1]{EL04}
\[
  \inf_{\substack{\cL\subset\dom(\fra) \\[0.2ex] \dim\cL=\kappa+n}}\;
  \max_{x\in\cL\setminus\{0\}}\;\;p(x) = \infty
\]
for $n>N$, i.e.\ for every $\cL\subset\dom(\fra)$ with $\dim\cL>\kappa+N$ one has
$\max_{x\in\cL\setminus\{0\}}p(x)=\infty$.
By \cite[Lemma~2.8]{EL04} the maximum is attained and therefore $p(x)=\infty$
for some $x\in\cL$.  However, this is a contradiction to the definition of $p$
in our case and hence $\lae<\infty$.
\end{proof}

In the next corollary we consider again the case when $A$ has compact resolvent.

\begin{corollary}\label{co:comp_res1}
Suppose that Assumption~{\rm\ref{assump}} is satisfied, let $\cM$ be the operator
as in Theorem~{\rm\ref{jsaness}}, let $\mu$ be as in \eqref{defmu}
and let $p$ and $\kappa_-$ be as in Definitions~{\rm\ref{def:rayleigh}}
and {\rm\ref{def:kappami}}.
Assume that $A$ has compact resolvent and that $\cH_1$ is infinite-dimensional.

Then $\kappa_-(S(\gamma))<\infty$ for every $\gamma\in(\mu,\infty)$.
Moreover, $\sigmaess(\cM)\cap(\mu,\infty)=\varnothing$ and hence $\lae=\infty$.
Further, $\sigma(\cM)\cap(\mu,\infty)$ consists of infinitely many eigenvalues
{\rm(}i.e.\ $N=\infty${\rm)},
which accumulate only at $\infty$, and \eqref{var1} holds for all $n\in\NN$.
\end{corollary}

\begin{proof}
It follows from Lemma~\ref{forms} and the proof of Corollary~\ref{compres}
that $S(\gamma)$ is bounded from below and has compact resolvent.
This implies that $\kappa<\infty$.
Moreover, $\sigmaess(\cM)\cap(\mu,\infty)=\varnothing$ by Corollary~\ref{compres}.
Finally, $N=\infty$ because otherwise $\lae<\infty$ by Theorem~\ref{finitevar}.
\end{proof}

The next simple lemma is used below and was proved in \cite[Lemma~3.5]{LLT02}
for bounded $B$.

\begin{lemma}\label{le:simple}
Let $x\in\dom(\fra)\backslash\{0\}$ such that $B^*x\ne0$ and let $\lambda\in\UU$.
With $y\defeq(D-\lambda)^{-1}B^*x$ we have
\[
  \det\bigl(\cM_{x,y}-\lambda\bigr)
  = \frac{\bigl\langle B^*x,(D-\lambda)^{-1}B^*x\bigr\rangle\,\frs(\lambda)[x]}{\|x\|^2\|y\|^2}\,.
\]
\end{lemma}

\begin{proof}
Clearly, $y\ne0$.  From the definition of $\cM_{x,y}$ we obtain
\begin{align*}
  & \|x\|^2\|y\|^2\det\bigl(\cM_{x,y}-\lambda\bigr)
  = \bigl(\fra[x]-\lambda\|x\|^2\bigr)\bigl(\frd[y]-\lambda\|y\|^2\bigr)
  + \langle y,B^*x\rangle\langle B^*x,y\rangle
  \\[0.5ex]
  &= \bigl(\fra[x]-\lambda\|x\|^2\bigr)\bigl\langle B^*x,(D-\lambda)^{-1}B^*x\bigr\rangle
  \\[0.5ex]
  &\quad + \bigl\langle(D-\lambda)^{-1}B^*x,B^*x\bigr\rangle
  \bigl\langle B^*x,(D-\lambda)^{-1}B^*x\bigr\rangle
  \\[0.5ex]
  &= \bigl\langle B^*x,(D-\lambda)^{-1}B^*x\bigr\rangle\frs(\lambda)[x],
\end{align*}
which proves the assertion.
\end{proof}

In the next proposition we consider the case when one of the inequalities
\eqref{condA1}, \eqref{condA2} is strict.  Then the index shift $\kappa$
is equal to $0$ for appropriate $\gamma_0$.

\begin{proposition}\label{pr:strictA}
Suppose that Assumption~{\rm\ref{assump}} is satisfied, let $\cM$ be the operator
as in Theorem~{\rm\ref{jsaness}},
let $a\in\RR$, $b\ge0$ be such that \eqref{domcons} is satisfied
and let $\mu$, $\mu_+$ be as in \eqref{defmu} and \eqref{defmupl}, respectively.
Assume that
\begin{equation}\label{strictcondA}
  b\delta_++b^2+a < 0 \qquad\text{or}\qquad
  \frac{\alpha_--\delta_+}{2} > b + \sqrt{\bigl(b\delta_++b^2+a\bigr)_+\!}\,.
\end{equation}
Then
\begin{equation}\label{chinequ}
  \delta_+<\mu<\mu_+\le\alpha_-,
\end{equation}
and for each $\gamma\in(\mu,\mu_+)$ there exists a $c>0$ such that
\begin{equation}\label{sunifpos}
  \frs(\gamma)[x]\ge c\|x\|^2, \qquad x\in\dom(\fra),
\end{equation}
and hence $\kappa_-(S(\gamma))=0$.
\end{proposition}

\begin{proof}
The inequalities in \eqref{chinequ} follow from Remark~\ref{re:pos_gap}\,(iii).
Let $\gamma\in(\mu,\mu_+)$ and $x\in\dom(\fra)\backslash\{0\}$.
We first show that $\frs(\gamma)[x]>0$.
If $B^*x=0$, then $\frs(\gamma)[x] = \fra[x]-\gamma\|x\|^2 > 0$
since $\gamma<\alpha_-$.
Now assume that $B^*x\ne0$.  Set $y\defeq(D-\gamma)^{-1}B^*x$.
From Proposition~\ref{numrange} we obtain
\[
  \lambda_-\binom{x}{y} \le \mu < \gamma < \mu_+ \le \lambda_+\binom{x}{y}.
\]
Now Lemma~\ref{le:simple} implies that
\begin{align*}
  0 &> \biggl(\gamma-\lambda_+\binom xy\biggr)\biggl(\gamma-\lambda_-\binom xy\biggr)
  = \det\bigl(\cM_{x,y}-\gamma\bigr)
  \\[0.5ex]
  &= \frac{\bigl\langle B^*x,(D-\gamma)^{-1}B^*x\bigr\rangle\,\frs(\gamma)[x]}{\|x\|^2\|y\|^2}\,.
\end{align*}
Since $\gamma>\delta_+$, this shows that $\frs(\gamma)[x]>0$.
Hence the operator $S(\gamma)$ is non-negative.

Further, $(\mu,\mu_+)\subset\rho(\cM)$ by Theorem~\ref{th:specinB}
and therefore $0\in\rho(S(\gamma))$ by Theorem~\ref{th:sess_equiv}.
This proves that $S(\gamma)$ is uniformly positive, i.e.\ \eqref{sunifpos} holds
and $\kappa_-(S(\gamma))=0$.
\end{proof}

In Theorem~\ref{finitevar2} below we prove a variational principle
with the functional $\lambda_+$.  To this end we need some lemmas
to rewrite $p(x)$ in terms of $\lambda_+$.

\begin{lemma}\label{lamlem1}
Let $x\in\dom(\fra)$ and assume that $B^*x\ne 0$.
If\, $\frs(\lambda)[x] \le 0$ for some $\lambda\in(\mu,\infty)$, then
\begin{equation}\label{lambdapl}
  \lambda_+\begin{pmatrix} x \\[0.5ex] (D-\lambda)^{-1}B^*x \end{pmatrix} \le \lambda.
\end{equation}
If\, $\frs(\lambda)[x]=0$, then there is equality in \eqref{lambdapl}.
\end{lemma}

\begin{proof}
Set $y\defeq(D-\lambda)^{-1}B^*x$, which is non-zero.  From Lemma~\ref{le:simple}
we obtain
\[
  \det\bigl(\cM_{x,y}-\lambda\bigr)
  = \frac{\bigl\langle B^*x,(D-\lambda)^{-1}B^*x\bigr\rangle\,\frs(\lambda)[x]}{\|x\|^2\|y\|^2}
  \ge 0
\]
since $\lambda>\delta_+$.  By Lemma~\ref{lamlem0} we have $\lambda_-\binom{x}{y} \le \mu < \lambda$,
which implies that $\lambda_+\binom{x}{y}\le\lambda$.
If $\frs(\lambda)[x]=0$, then $\det(\cM_{x,y}-\lambda)=0$ and hence $\lambda_+\binom{x}{y}=\lambda$.
\end{proof}

\begin{lemma}\label{lamlem2}
Let $x\in\dom(\fra)\backslash\{0\}$.
\begin{myenum}
\item
If\, $\frs(\lambda_0)[x]= 0$ for some $\lambda_0\in(\mu,\infty)$,
i.e.\ $\lambda_0=p(x)$, then
\[
  \lambda_\pm\binom{x}{y}\in\RR \qquad\text{for all}\quad y\in\dom(\frd)\backslash\{0\}
\]
and
\begin{equation}\label{lambda0min}
  \min_{y\in\dom(\frd)\setminus\{0\}}\lambda_+\binom{x}{y} = \lambda_0.
\end{equation}
\item
If\, $\frs(\lambda)[x]<0$ for all $\lambda\in(\mu,\infty)$,
i.e.\ $p(x)=-\infty$, then
\begin{equation}\label{lambda0inf}
  \inf_{y\in\dom(\frd)\setminus\{0\}}\Re\lambda_+\binom{x}{y} \le \mu.
\end{equation}
\end{myenum}
\end{lemma}

\begin{proof}
If $B^*x=0$, then $\frs(\lambda)[x] = \fra[x]-\lambda\|x\|^2$ and
\[
  \lambda_+\binom{x}{y} = \max\biggl\{\frac{\fra[x]}{\|x\|^2},\frac{\frd[y]}{\|y\|^2}\biggr\}
  \qquad\text{for all}\quad y\in\dom(\frd)\backslash\{0\}.
\]
Since $\frd[y]/\|y\|^2\le\delta_+\le\mu$, the assertion follows in both cases (i) and (ii).

For the rest of the proof we assume that $B^*x\ne 0$.

(i) Suppose that $\frs(\lambda_0)[x]=0$ for some $\lambda_0\in(\mu,\infty)$.
For any $y\in\dom(\frd)\backslash\{0\}$ we have
\begin{align*}
  |\langle y,B^*x\rangle|^2
  &= \bigl|\bigl\langle(\lambda_0-D)^{\frac12}y,(\lambda_0-D)^{-\frac12}B^*x\bigr\rangle\bigr|^2
  \\[0.5ex]
  &\le \bigl\|(\lambda_0-D)^{\frac12}y\bigr\|^2\,\bigl\|(\lambda_0-D)^{-\frac12}B^*x\bigr\|^2
  \\[0.5ex]
  &= (\frd-\lambda_0)[y]\bigl\langle(D-\lambda_0)^{-1}B^*x,B^*x\bigr\rangle
\end{align*}
and hence
\begin{align}
  \det\bigl(\cM_{x,y}-\lambda_0\bigr)
  &= \frac{(\fra-\lambda_0)[x]}{\|x\|^2}\cdot\frac{(\frd-\lambda_0)[y]}{\|y\|^2}
  + \frac{|\langle y,B^*x\rangle|^2}{\|x\|^2\|y\|^2}
  \notag\\[0.5ex]
  &\le \frac{(\fra-\lambda_0)[x]\cdot(\frd-\lambda_0)[y]
  + (\frd-\lambda_0)[y]\bigl\langle(D-\lambda_0)^{-1}B^*x,B^*x\bigr\rangle}{\|x\|^2\|y\|^2}
  \notag\\[0.5ex]
  &= \frac{\frs(\lambda_0)[x]\cdot(\frd-\lambda_0)[y]}{\|x\|^2\|y\|^2}
  = 0.
  \label{detle0}
\end{align}
Since $\det(\cM_{x,y}-\lambda)$ is a monic quadratic polynomial in $\lambda$ with
real coefficients, the inequality in \eqref{detle0} implies that
its zeros $\lambda_\pm\binom{x}{y}$ are real and that
\[
  \lambda_+\binom{x}{y} \ge \lambda_0.
\]
This, together with Lemma~\ref{lamlem1}, proves \eqref{lambda0min}.

(ii)
Now assume that $\frs(\lambda)[x]<0$ for all $\lambda\in(\mu,\infty)$.
For each $\lambda\in(\mu,\infty)$ we obtain from Lemma~\ref{lamlem1} that there
exists a $y\in\dom(\frd)\backslash\{0\}$ such that $\lambda_+\binom{x}{y}\le\lambda$,
which implies \eqref{lambda0inf}.
\end{proof}

The next theorem contains a variational principle with the functional $\lambda_+$
that is connected with the quadratic numerical range of $\cM$.
It follows immediately from Theorem~\ref{finitevar}, Remark~\ref{rem:rayleighinf}
and Lemma~\ref{lamlem2}.
Similar results were obtained in \cite[Theorem~5.3]{LLMT05}
for bounded $A$, $B$, $D$ and in \cite[Theorem~4.2]{LLT02}
when $B$ is bounded and $W^2(\cM)$ consists of two separated components
and hence $\sigma(\cM)$ is real.

\begin{theorem}\label{finitevar2}
Suppose that Assumption~{\rm\ref{assump}} and \eqref{condkappa} are satisfied.
Let $\gamma_0$, $\kappa$, $\lae$, $N$, $(\lambda_j)_{j=1}^N$ be
as in Theorem~{\rm\ref{finitevar}}.
Then
\begin{align*}
  \lambda_n &= \min_{\substack{\cL\subset\dom(\fra) \\[0.2ex] \dim\cL=\kappa+n}}\;
  \max_{x\in\cL\setminus\{0\}}\;\;
  \inf_{y\in\dom(\frd)\setminus\{0\}} \Re\lambda_+\binom{x}{y}
  \\[1ex]
  &= \max_{\substack{\cL\subset\mathcal{H} \\[0.2ex] \dim\cL=\kappa+n-1}}
  \inf_{\substack{x\in\dom(\fra)\setminus\{0\} \\[0.2ex] x\perp\cL}}\;\;
  \inf_{y\in\dom(\frd)\setminus\{0\}} \Re\lambda_+\binom{x}{y}
\end{align*}
for $n\in\NN$, $n\le N$.
Moreover, if\, $N$ is finite and $\cH_1$ is infinite-dimensional,
then $\lae<\infty$ and
\begin{align*}
  \lae &= \min_{\substack{\cL\subset\dom(\fra) \\[0.2ex] \dim\cL=\kappa+n}}\;
  \max_{x\in\cL\setminus\{0\}}\;\;
  \inf_{y\in\dom(\frd)\setminus\{0\}} \Re\lambda_+\binom{x}{y}
  \\[1ex]
  &= \max_{\substack{\cL\subset\mathcal{H} \\[0.2ex] \dim\cL=\kappa+n-1}}
  \inf_{\substack{x\in\dom(\fra)\setminus\{0\} \\[0.2ex] x\perp\cL}}\;\;
  \inf_{y\in\dom(\frd)\setminus\{0\}} \Re\lambda_+\binom{x}{y}
\end{align*}
for $n>N$.
\end{theorem}

\section{Eigenvalue estimates and asymptotics}\label{estimates}

\noindent
In this section we prove estimates for certain real eigenvalues of $\cM$.
In particular, we compare these eigenvalues with eigenvalues of $A$.
To this end, we denote by $\nu_1\le\nu_2\le\cdots$ the eigenvalues of $A$
that lie below $\min\sigmaess(A)$ counted according to multiplicities.
If $A$ has only finitely many eigenvalues, say $M$, below its essential spectrum
and $\cH_1$ is infinite-dimensional, then set $\nu_k\defeq\min\sigmaess(A)$ for $k>M$.

In the case when $A$ has compact resolvent we also show asymptotic estimates.
The following estimates are analogous to those found for upper dominant
self-adjoint matrices; see \cite[Section~4.1]{KLT04}.
The first inequality in \eqref{est1} below and \eqref{est2} were proved
in \cite[Theorem~5.2]{LLT02} under the extra assumption that $B$ is bounded
and $W^2(\cM)$ consists of two separated components and hence $\sigma(\cM)$ is real.

\begin{corollary}\label{ests}
Suppose that Assumption~{\rm\ref{assump}} and \eqref{condkappa} are satisfied.
Let $\gamma_0$, $\kappa$, $\lae$, $N$, $(\lambda_j)_{j=1}^N$ be
as in Theorem~{\rm\ref{finitevar}}.
If\, $N$ is finite, then set $\lambda_n\defeq\lae$ for $n>N$.
Moreover, let $\nu_k$ be as above.  Then
\begin{equation}\label{est1}
  \frac{\nu_{\kappa+n}+\delta_+}{2} + \sqrt{\biggl(\biggl(\frac{\nu_{\kappa+n}-\delta_+}{2}\biggr)^2
  - b\nu_{\kappa+n} - a\biggr)_+}
  \le \lambda_n \le \nu_{\kappa+n}
\end{equation}
for $n\in\NN$ such that $\kappa+n\le\dim\cH_1$.

Assume, in addition, that $D$ is bounded, set $\delta_-\defeq\min\sigma(D)$,
and let $\hat{a}\in\RR$ and $\hat{b}\ge0$ be such that
\begin{equation}\label{Blowest}
  \|B^*x\|^2\ge\hat{a}\|x\|^2 + \hat{b}\fra[x], \qquad x\in\dom(\fra).
\end{equation}
Then
\begin{equation}\label{discr_est}
  \biggl(\frac{\nu_{\kappa+n}-\delta_-}{2}\biggr)^2 - \hat{b}\nu_{\kappa+n} - \hat{a} \ge 0
\end{equation}
and
\begin{equation}\label{est2}
  \lambda_n \le \frac{\nu_{\kappa+n}+\delta_-}{2}
  + \sqrt{\biggl(\frac{\nu_{\kappa+n}-\delta_-}{2}\biggr)^2 - \hat{b}\nu_{\kappa+n} - \hat{a}\,}
\end{equation}
for $n\in\NN$ such that $\kappa+n\le\dim\cH_1$.
\end{corollary}

\begin{proof}
Throughout the proof let $n\in\NN$ such that $\kappa+n\le\dim\cH_1$.
First we show that
\begin{equation}\label{plea}
  p(x) \le \frac{\fra[x]}{\|x\|^2}\,, \qquad x\in\dom(\fra)\backslash\{0\}.
\end{equation}
If $p(x)=-\infty$, then the statement is trivial.
Otherwise, we have $p(x)>\mu\ge\delta_+$, and therefore
\begin{align*}
  0 &= \frs\bigl(p(x)\bigr)[x]
  = \fra[x] - p(x)\|x\|^2 + \bigl\langle\bigl(D-p(x)\bigr)^{-1}B^*x,B^*x\bigr\rangle
  \\[0.5ex]
  &\le \fra[x] - p(x)\|x\|^2,
\end{align*}
which implies \eqref{plea}.  Now \eqref{var1}, \eqref{var2} and the standard variational principle
for self-adjoint operators imply that
\[
  \lambda_n = \min_{\substack{\cL\subset\dom(\fra) \\[0.2ex] \dim\cL=\kappa+n}}\;
  \max_{x\in\cL\setminus\{0\}}\;\;p(x)
  \le \min_{\substack{\cL\subset\dom(\fra) \\[0.2ex] \dim\cL=\kappa+n}}\;
  \max_{x\in\cL\setminus\{0\}}\;\;\frac{\fra[x]}{\|x\|^2}
  = \nu_{\kappa+n},
\]
which shows the second inequality in \eqref{est1}.

If the expression
\begin{equation}\label{discr_neg}
  \biggl(\frac{\nu_{\kappa+n}-\delta_+}{2}\biggr)^2
  - b\nu_{\kappa+n} - a
\end{equation}
is negative, then the left-hand side of \eqref{est1} is equal to $(\nu_{\kappa+n}+\delta_+)/2$,
which, by Lemma~\ref{le:nott}, satisfies
\[
  \frac{\nu_{\kappa+n}+\delta_+}{2} \le \delta_++b+\sqrt{b\delta_++b^2+a}
  = \mu < \lambda_n.
\]
Hence the first inequality in \eqref{est1} is proved in this case.

Now assume that the expression in \eqref{discr_neg} is non-negative.
Let $x\in\dom(\fra)\backslash\{0\}$ and $\lambda>\mu$.  From \eqref{domcons} we obtain
\begin{align*}
  \frs(\lambda)[x] &= \fra[x] - \lambda\|x\|^2 - \bigl\|(\lambda-D)^{-\frac12}B^*x\bigr\|^2
  \\[0.5ex]
  &\ge \fra[x] - \lambda\|x\|^2 - \frac{\|B^*x\|^2}{\lambda-\delta_+}
  \\[0.5ex]
  &\ge \fra[x] - \lambda\|x\|^2 - \frac{b\fra[x]+a\|x\|^2}{\lambda-\delta_+}
  \\[0.5ex]
  &= \frac{(\lambda-\delta_+-b)\fra[x]-(\lambda^2-\delta_+\lambda+a)\|x\|^2}{\lambda-\delta_+}\,.
\end{align*}
It follows from the standard variational principle for self-adjoint operators that,
for every $(\kappa+n)$-dimensional subspace $\cL\subset\dom(\fra)$, there exists
an $x_\cL\in\cL$ such that $\|x_\cL\|=1$ and $\fra[x_\cL]\ge\nu_{\kappa+n}$.
Then
\begin{equation}\label{234}
\begin{aligned}
  \frs(\lambda)[x_\cL] &\ge \frac{1}{\lambda-\delta_+}\Bigl[
  (\lambda-\delta_+-b)\nu_{\kappa+n}-\lambda^2+\delta_+\lambda-a\Bigr]
  \\[0.5ex]
  &= -\frac{1}{\lambda-\delta_+}\Bigl[\lambda^2-(\nu_{\kappa+n}+\delta_+)\lambda
  + \delta_+\nu_{\kappa+n} + b\nu_{\kappa+n} + a\Bigr].
\end{aligned}
\end{equation}
Since the expression in \eqref{discr_neg} is non-negative, the polynomial
in $\lambda$ within the square brackets has real zeros.
The larger of these zeros is equal to the left-hand side of \eqref{est1},
which we denote by $\mu_n$.  From \eqref{234} we obtain $\frs(\mu_n)[x_\cL]\ge0$
and hence $p(x_\cL)\ge\mu_n$ since $\frs$
satisfies the condition \VM.
Now \eqref{var1}, \eqref{var2} imply that
\[
  \lambda_n = \inf_{\substack{\cL\subset\dom(\fra) \\[0.2ex] \dim\cL=\kappa+n}}\;
  \max_{x\in\cL\setminus\{0\}}\;\;p(x)
  \ge \inf_{\substack{\cL\subset\dom(\fra) \\[0.2ex] \dim\cL=\kappa+n}}\;p(x_\cL)
  \ge \mu_n,
\]
which is the first inequality in \eqref{est1}.

Now assume that $D$ is bounded and that \eqref{Blowest} is satisfied.
Let $\eps>0$ be arbitrary.
By the standard variational principle applied to $A$
there exists an $\cL_0\subset\dom(\fra)$ with $\dim\cL_0=\kappa+n$ such
that $\fra[x]\le(\nu_{\kappa+n}+\eps)\|x\|^2$ for all $x\in\cL_0$
(if $\nu_{\kappa+n}$ is an eigenvalue, we could choose $\eps=0$).
From \eqref{var1} and \eqref{var2} we obtain that
\begin{equation}\label{partsubsp}
  \lambda_n = \inf_{\substack{\cL\subset\dom(\fra) \\[0.2ex] \dim\cL=\kappa+n}}\;
  \max_{x\in\cL\setminus\{0\}}\;\;p(x)
  \le \max_{x\in\cL_0\setminus\{0\}}\;\;p(x).
\end{equation}
For $x\in\cL_0\backslash\{0\}$ and $\lambda\in(\mu,\infty)$ we have
\begin{align}
  \frs(\lambda)[x] &= \fra[x] - \lambda\|x\|^2 - \bigl\langle(\lambda-D)^{-1}B^*x,B^*x\bigr\rangle
  \notag\\[0.5ex]
  &\le \fra[x] - \lambda\|x\|^2 - \frac{\|B^*x\|^2}{\lambda-\delta_-}
  \notag\displaybreak[0]\\[0.5ex]
  &\le \fra[x] - \lambda\|x\|^2 - \frac{\hat b\fra[x]+\hat a\|x\|^2}{\lambda-\delta_-}
  \notag\displaybreak[0]\\[0.5ex]
  &\le \frac{(\lambda-\delta_--\hat b)(\nu_{\kappa+n}+\eps)
  -\lambda^2+\delta_-\lambda-\hat a}{\lambda-\delta_-}\|x\|^2
  \notag\\[0.5ex]
  &= -\frac{\lambda^2-(\nu_{\kappa+n}+\eps+\delta_-)\lambda
  +\delta_-(\nu_{\kappa+n}+\eps)+\hat b(\nu_{\kappa+n}+\eps)+\hat a}{\lambda-\delta_-}\|x\|^2.
  \label{estsabove}
\end{align}
Let $\hat\mu_{n,\eps,\pm}$ be the zeros of the polynomial in $\lambda$ in the
numerator of the fraction in \eqref{estsabove}, i.e.\
\[
  \hat\mu_{n,\eps,\pm} = \frac{\nu_{\kappa+n}+\eps+\delta_-}{2}
  \pm \sqrt{\biggl(\frac{\nu_{\kappa+n}+\eps-\delta_-}{2}\biggr)^2-\hat b(\nu_{\kappa+n}+\eps)-\hat a\,}.
\]
If these zeros are non-real or $\hat\mu_{n,\eps,+}\le\mu$, then $\frs(\lambda)[x]<0$
for all $\lambda\in(\mu,\infty)$ and hence $p(x)=-\infty$ for all $x\in\cL_0\backslash\{0\}$.
This, together with \eqref{partsubsp}, would imply that $\lambda_n=-\infty$, a contradiction.
Therefore $\hat\mu_{n,\eps,\pm}\in\RR$ and $\hat\mu_{n,\eps,+}>\mu$.
In particular,
\begin{equation}\label{discr_eps}
  \biggl(\frac{\nu_{\kappa+n}+\eps-\delta_-}{2}\biggr)^2-\hat b(\nu_{\kappa+n}+\eps)-\hat a \ge 0.
\end{equation}
Relation \eqref{estsabove} yields that $\frs(\lambda)[x]<0$
for all $\lambda\in(\hat\mu_{n,\eps,+},\infty)$ and hence $p(x)\le\hat\mu_{n,\eps,+}$
for all $x\in\cL_0\backslash\{0\}$.  This, together with \eqref{partsubsp} implies
that $\lambda_n\le\hat\mu_{n,\eps,+}$.
If we take the limit as $\eps\to0$ in the latter inequality and in \eqref{discr_eps},
we obtain \eqref{discr_est} and \eqref{est2}.
\end{proof}

Under the extra assumption that $A$ has compact resolvent we can obtain
asymptotic estimates for the eigenvalues of $\cM$ that lie in $(\mu,\infty)$.
Analogous estimates for self-adjoint block operator matrices were shown
in \cite[Corollary~4.4]{KLT04}.

\begin{corollary}\label{compres2}
Suppose that Assumption~{\rm\ref{assump}} is satisfied, let $\cM$ be the operator
as in Theorem~{\rm\ref{jsaness}}, let $\mu$ be as in \eqref{defmu} and
let $\nu_k$ be as at the beginning of the section.
Assume that $A$ has compact resolvent and that $\cH_1$ is infinite-dimensional.
Then $(\mu,\infty)\cap\sigma(\cM)$ consists of a sequence of
eigenvalues that tends to $\infty$.
Let $\gamma_0\in(\mu,\infty)$, let $\kappa$ be as in \eqref{defkappa} and
let $(\lambda_n)_{n=1}^\infty$ be the eigenvalues in $[\gamma_0,\infty)$.
Then
\begin{equation}\label{asympest1}
  \nu_{\kappa+n} - b - \frac{b^2+b\delta_++a}{\nu_{\kappa+n}-\delta_+}
  + \cO\biggl(\frac{1}{\nu_{\kappa+n}^2}\biggr)
  \le \lambda_n \le \nu_{\kappa+n},
  \qquad n\to\infty.
\end{equation}
If, in addition, $D$ is bounded with $\delta_-$, $\hat{a}$ and $\hat{b}$
as in Corollary~{\rm\ref{ests}}, then
\begin{equation}\label{asympest2}
  \lambda_n \le \nu_{\kappa+n} - \hat{b}
  - \frac{\hat{b}^2+\hat{b}\delta_-+\hat{a}}{\nu_{\kappa+n}-\delta_-}
  + \cO\biggl(\frac{1}{\nu_{\kappa+n}^2}\biggr),
  \qquad n\to\infty.
\end{equation}
\end{corollary}

\begin{proof}
The first statements follow from Corollary~\ref{co:comp_res1}.
For the estimates we use Corollary~\ref{ests}, which for sufficiently large $n$ yields
\begin{align*}
  \lambda_n &\ge \frac{\nu_{\kappa+n}+\delta_+}{2}
  + \sqrt{\biggl(\frac{\nu_{\kappa+n}-\delta_+}{2}\biggr)^2-b\nu_{\kappa+n}-a\,}
  \\[0.5ex]
  &= \frac{\nu_{\kappa+n}+\delta_+}{2} + \biggl(\frac{\nu_{\kappa+n}-\delta_+}{2}\biggr)
  \sqrt{1-\biggl(\frac{2}{\nu_{\kappa+n}-\delta_+}\biggr)^2\bigl(b\nu_{\kappa+n}+a\bigr)\,}
  \displaybreak[0]\\[0.5ex]
  &= \nu_{\kappa+n}
  - \frac{1}{2}\biggl(\frac{2}{\nu_{\kappa+n}-\delta_+}\biggr)\bigl(b\nu_{\kappa+n}+a\bigr)
  - \frac{1}{8}\biggl(\frac{2}{\nu_{\kappa+n}-\delta_+}\biggr)^3\bigl(b\nu_{\kappa+n}+a\bigr)^2
  - \ldots
  \displaybreak[0]\\[0.5ex]
  &= \nu_{\kappa+n} - \frac{b\nu_{\kappa+n}+a}{\nu_{\kappa+n}-\delta_+}
  - \frac{(b\nu_{\kappa+n}+a)^2}{(\nu_{\kappa+n}-\delta_+)^3}
  + \cO\biggl(\frac{1}{\nu_{\kappa+n}^2}\biggr)
  \\[0.5ex]
  &= \nu_{\kappa+n} - b - \frac{b\delta_++b^2+a}{\nu_{\kappa+n}-\delta_+}
  + \cO\biggl(\frac{1}{\nu_{\kappa+n}^2}\biggr).
\end{align*}
If $D$ is bounded, we obtain from \eqref{est2} in a similar way that
\begin{align*}
  \lambda_n &\le \frac{\nu_{\kappa+n}+\delta_-}{2}
  + \sqrt{\biggl(\frac{\nu_{\kappa+n}-\delta_-}{2}\biggr)^2 - \hat{b}\nu_{\kappa+n} - \hat{a}\,}
  \\[0.5ex]
  &= \nu_{\kappa+n} - \hat{b} - \frac{\hat{b}\delta_-+\hat{b}^2+\hat{a}}{\nu_{\kappa+n}-\delta_-}
  + \cO\biggl(\frac{1}{\nu_{\kappa+n}^2}\biggr).
\end{align*}
\end{proof}

\section{Spectrum of positive type}\label{krein}

\noindent
In this section we consider properties of the operator $\cM$ considered
in the Krein space $\cK\defeq\cH_1\oplus\cH_2$ equipped with the
indefinite inner product \eqref{indef_pr}.
Recall that a point $\lambda\in\sigmaapp(\cM)$ is called
\emph{spectral point of positive type} if for every sequence $(x_n,y_n)^T\in\dom(\cM)$
such that
\begin{equation}\label{defpostype1}
  (\cM-\lambda)\binom{x_n}{y_n} \to 0 \qquad\text{and}\qquad
  \|x_n\|^2+\|y_n\|^2 = 1
\end{equation}
one has
\begin{equation}\label{defpostype2}
  \liminf_{n\to\infty} \biggl[\binom{x_n}{y_n},\binom{x_n}{y_n}\biggr] > 0.
\end{equation}

\medskip

In the next theorem we consider spectral points in the interval $(\mu,\infty)$.
For bounded operators $\cM$ the result was shown in \cite[Theorem~3.1]{LLMT05}.

\begin{theorem}\label{th:postype}
Suppose that Assumption~{\rm\ref{assump}} is satisfied, let $\cM$ be the operator
as in Theorem~{\rm\ref{jsaness}} and let $\mu$ be as in \eqref{defmu}.
Then all points from $\sigma(\cM)\cap(\mu,\infty)$
are spectral points of positive type.
\end{theorem}

\begin{proof}
Let $\lambda\in\sigma(\cM)\cap(\mu,\infty)$ and let $(x_n,y_n)^T\in\dom(\cM)$
such that \eqref{defpostype1} holds.
It follows from Lemma~\ref{le:approx_es} that $\frs(\lambda)[x_n]\to0$,
$\liminf_{n\to\infty}\|x_n\|>0$ and that $\|B^*x_n\|$ is bounded.
Since $\frs$ satisfies the condition \VM{} by Lemma~\ref{vmlem},
there exist $\eps,\delta>0$ such that \eqref{VMimpl} holds.
Hence there exists an $N\in\NN$ so that
\[
  \frs'(\lambda)[x_n] \le -\delta\|x_n\|^2, \qquad n\ge N.
\]
Lemma~\ref{le:approx_es}\,(ii) implies that $y_n=(D-\lambda)^{-1}B^*x_n+w_n$
with $w_n\to0$.
This, together with \eqref{frs2}, yields
\begin{align*}
  & \biggl[\binom{x_n}{y_n},\binom{x_n}{y_n}\biggr]
  = \|x_n\|^2 - \|y_n\|^2
  \\[0.5ex]
  &= \|x_n\|^2 - \bigl\|(D-\lambda)^{-1}B^*x_n+w_n\bigr\|^2
  \displaybreak[0]\\[0.5ex]
  &= \|x_n\|^2 - \bigl\|(D-\lambda)^{-1}B^*x_n\bigr\|^2
  - 2\Re\bigl\langle(D-\lambda)^{-1}B^*x_n,w_n\bigr\rangle - \|w_n\|^2
  \\[0.5ex]
  &= -\frs'(\lambda)[x_n] + \rmo(1)
  \\[0.5ex]
  &\ge \delta\|x_n\|^2 + \rmo(1).
\end{align*}
Since $\liminf_{n\to\infty}\|x_n\|^2>0$, we obtain
\[
  \liminf_{n\to\infty} \biggl[\binom{x_n}{y_n},\binom{x_n}{y_n}\biggr]
  \ge \delta\,\liminf_{n\to\infty}\|x_n\|^2 > 0,
\]
which shows that $\lambda$ is a spectral point of positive type.
\end{proof}

It follows from this theorem and \cite[\S8]{AJT05}
(cf.\ also \cite[Theorem~3.1]{LMM97} for bounded operators)
that the operator $\cM$ has a local spectral function of positive type on
the interval $(\mu,\infty)$.

In the following proposition we consider again the situation from
Proposition~\ref{pr:strictA}, namely that in \eqref{condA1} or \eqref{condA2}
strict inequality holds.

\begin{proposition}\label{pr:kreinnonneg}
Suppose that Assumption~{\rm\ref{assump}} is satisfied, let $\cM$ be the operator
as in Theorem~{\rm\ref{jsaness}} and assume that the condition
in \eqref{strictcondA} is satisfied.
For $\gamma\in(\mu,\mu_+)$ the operator $\cM-\gamma$ is non-negative
in the Krein space $\cK$ and $\gamma\in\rho(\cM)$.
\end{proposition}

\begin{proof}
Let $(x,y)^T\in\dom(\cM)$.  From \eqref{connection_T_form} we obtain
\begin{align}
  \biggl[(\cM-\gamma)\binom{x}{y},\binom{x}{y}\biggr]
  &= \biggl\langle(\cM-\gamma)\binom{x}{y},\binom{x}{-y}\biggr\rangle
  \notag\\[0.5ex]
  &= \fra[x] - \gamma\|x\|^2 + \langle y,B^*x\rangle + \langle B^*x,y\rangle
  - \langle Dy,y\rangle + \gamma\|y\|^2
  \notag\\[0.5ex]
  &= (\fra-\gamma)[x] - (\frd-\gamma)[y] + 2\Re\langle y,B^*x\rangle
  \notag\\[0.5ex]
  &\ge (\fra-\gamma)[x] - (\frd-\gamma)[y] - 2\bigl|\langle y,B^*x\rangle\bigr|.
  \label{kreinform}
\end{align}
If $x=0$ or $y=0$, then the expression in \eqref{kreinform} is non-negative
since $\delta_+<\gamma<\alpha_-$.
Now assume that $x,y\ne0$.
It follows from Proposition~\ref{numrange} that $\lambda_\pm\binom{x}{y}\in\RR$ and
\[
  \lambda_-\binom{x}{y} \le \mu < \gamma < \mu_+ \le \lambda_+\binom{x}{y}.
\]
Therefore
\begin{align*}
  & \frac{(\fra-\gamma)[x]\cdot(\frd-\gamma)[y]+|\langle y,B^*x\rangle|^2}{\|x\|^2\|y\|^2}
  = \det(\cM_{x,y}-\gamma)
  \\[0.5ex]
  &= \biggl(\gamma-\lambda_+\binom{x}{y}\biggr)\biggl(\gamma-\lambda_-\binom{x}{y}\biggr)
  < 0.
\end{align*}
If we combine this with the inequality in \eqref{kreinform}, we obtain
\[
  \biggl[(\cM-\gamma)\binom{x}{y},\binom{x}{y}\biggr]
  \ge (\fra-\gamma)[x] + (\gamma-\frd)[y] - 2\sqrt{(\fra-\gamma)[x]}\cdot\sqrt{(\gamma-\frd)[y]}
  \ge 0.
\]
Relation \eqref{sunifpos} implies also that $0\in\rho(S(\lambda))$,
which, by Theorem~\ref{specequiv}, yields that $\gamma\in\rho(\cM)$.
\end{proof}

If \eqref{strictcondA} is satisfied and $D$ is bounded, then,
by Proposition~\ref{pr:kreinnonneg}, all assumptions of \cite[Theorem~3.2]{MS96}
are satisfied with $\gamma\in(\mu,\mu_+)$.
The latter theorem implies, e.g.\ that the spectral subspaces corresponding
to $(\gamma,\infty)$ and $(-\infty,\gamma)$ are maximal uniformly positive
and negative, respectively.

A self-adjoint $T$ operator in a Krein space is called \emph{definitisable}
if $\rho(T)\ne\varnothing$ and there exists a real polynomial $p$ such that
\[
  [p(T)x,x] \ge 0 \qquad\text{for all}\;\; x\in\dom(T^{\deg p});
\]
see, e.g.\ \cite[p.~10]{L82}.
In the next theorem we consider again the situation when $A$ has compact resolvent.

\begin{theorem}\label{th:definitisable}
Suppose that Assumption~{\rm\ref{assump}} is satisfied and that $A$ has compact resolvent.
Then $\cM$ is definitisable, and hence the non-real spectrum of $\cM$ is finite.
\end{theorem}

\begin{proof}
That $\rho(\cM)\ne\varnothing$ follows from Theorem~\ref{th:specinB}.
Choose $\alpha_0$ so large that the second inequality in \eqref{strictcondA}
is satisfied with $\alpha_-$ replaced by $\alpha_0$.
Let $\cL$ be the spectral subspace for $A$ corresponding to $[\alpha_0,\infty)$,
let $\hat A$ be the restriction of $A$ to $\cL\cap\dom(A)$.
Moreover, let $P$ be the orthogonal projection in $\cH_1$ onto $\cL$, set
$\hat B\defeq PB$ and
\[
  \hat\cM_0 \defeq \begin{pmatrix} \hat A & \hat B \\ -\hat B^* & D \end{pmatrix},
\]
and let $\hat\cM$ be the closure of $\hat\cM_0$;
the operator $\hat\cM_0$ is understood as an operator in $\hat\cK\defeq\cL\oplus\cH_2$.
It is not difficult to see that Assumption~\ref{assump} is satisfied for $\hat\cM_0$
and that \eqref{domcons} holds with the same $a$ and $b$.
Clearly, $\hat\alpha_-\defeq\min\sigma(\hat A)\ge\alpha_0$ and therefore
the second inequality in \eqref{strictcondA} is satisfied with $\alpha_-$
replaced by $\hat\alpha_-$.
Proposition~\ref{pr:kreinnonneg} applied to $\hat\cM$ yields that $\hat\cM-\gamma$
is non-negative in $\hat\cK$ for some $\gamma\in\RR$.
Since $\hat\cK$ is finite co-dimensional in $\cK$,
this shows that $\cM-\gamma$ has finitely many squares, i.e.\ the form
$[(\cM-\gamma)\,\cdot\,,\cdot]$ is non-negative on a subspace with finite codimension.
By \cite[pp.~11--12]{L82} this implies that $\cM$ is definitisable.
It follows from \cite[Proposition~II.2.1 (p.~28)]{L82} that hence the non-real
spectrum of $\cM$ is finite.
\end{proof}

\section{Examples}\label{exs}

\noindent
In this section we consider two examples where the entries of the
block operator matrix $\cM_0$ are differential or multiplication operators.
The first example was studied in \cite{LLT02} for bounded $w$ and in \cite{JT02,LMM90}
in the one-dimensional case.

\begin{example}
Let $n\in\NN$ and let $\Omega\subset\RR^n$ be an arbitrary bounded domain
(we do not assume any smoothness of the boundary of $\Omega$).
Moreover, let $u\in L^\infty(\Omega)$ and $w\in L^p(\Omega)$ where
\begin{alignat*}{2}
  & p=1 \quad && \text{if} \; n=1, \\[0.5ex]
  & p>1 \quad && \text{if} \; n=2, \\[0.5ex]
  & p=\frac{n}{2} && \text{if} \; n\ge3,
\end{alignat*}
and assume that $u$ is real-valued and $w\ge0$.
Let $\cH_1=\cH_2=L^2(\Omega)$ and consider the operators $\cM_0$ and $\cM$
where $A=-\Delta$ with Dirichlet boundary conditions,
i.e.\ $A$ is the operator corresponding to the form
\[
  \fra[y] = \int_\Omega |\nabla y|^2, \qquad y\in\dom(\fra) = H^1_0(\Omega),
\]
and where $B$ and $D$ are the multiplication operators with the functions $\sqrt{w}$
and $u$, respectively.

By the Sobolev embedding theorem (see, e.g.\ \cite[Theorem~4.12]{AdamsFournier})
one has the continuous embedding $H^1_0(\Omega)\subset L^q(\Omega)$ where
$q=\infty$ if $n=1$; $q<\infty$ arbitrary if $n=2$; and $q=2n/(n-2)$ if $n\ge3$.
Since $\frac{1}{p}+\frac{2}{q}=1$ (where in the case $n=2$ one chooses $q$
accordingly for a given $p$), H\"older's inequality yields
\[
  \|B^*y\|^2 = \int_\Omega w|y|^2 \le \|w\|_{L^p(\Omega)}\|y\|_{L^q(\Omega)}^2,
  \qquad y\in L^q(\Omega).
\]
Therefore
\[
  \dom(\fra) = H^1_0(\Omega) \subset L^q(\Omega) \subset \dom(B^*),
\]
which shows that Assumption~\ref{assump}.(I) is satisfied;
note that $D$ is bounded.

With $C_{\rm Sob}$ and $C_{\rm Poinc}$ denoting the constants in
the Sobolev and the Poincar\'e inequalities, respectively,
we obtain
\begin{align*}
  \|B^*y\|^2 &\le \|w\|_{L^p(\Omega)}\|y\|_{L^q(\Omega)}^2
  \le C_{\rm Sob}^2\|w\|_{L^p(\Omega)}\|y\|_{H_0^1(\Omega)}^2
  \\[0.5ex]
  &\le C_{\rm Sob}^2C_{\rm Poinc}^2\|w\|_{L^p(\Omega)}\fra[y]
\end{align*}
for $y\in H_0^1(\Omega)$, which yields a possible choice for $b$ where $a=0$.

The Schur complement corresponds to the form
\[
  \frs(z)[y] = \int_\Omega \biggl(|\nabla y|^2 + \Bigl(-z+\frac{w}{u-z}\Bigr)|y|^2\biggr),
  \qquad y\in\dom(\frs(z))=H^1_0(\Omega),
\]
for $\dist(z,\essran u)>b_0$.  As an operator it acts like
\[
  S(z)y = -\Delta y+\Bigl(-z+\frac{w}{u-z}\Bigr)y.
\]

The operator $A$ has compact resolvent; let $\nu_1\le\nu_2\le\cdots$
be its eigenvalues in non-decreasing order and set
$\delta_-\defeq\essinf u$, $\delta_+\defeq\esssup u$.
Proposition~\ref{pr:sessALMS} implies that $\sigmaess(\cM)\subset[\delta_-,\delta_++b_0]$,
and Theorem~\ref{th:specinB} gives an enclosure for $\sigma(\cM)$.
Moreover, Theorem~\ref{th:definitisable} shows that the non-real spectrum is finite,
and Corollaries~\ref{co:comp_res1}, \ref{ests} and \ref{compres2} yield
that $\sigma(\cM)\cap(\mu,\infty)$ consists of a sequence of eigenvalues
that tends to $\infty$ and satisfies \eqref{est1} and \eqref{asympest1}.

If $p>n/2$ when $n\ge3$ and $p$ as above when $n=1,2$,
then the embedding $H^1_0(\Omega)\to L^q(\Omega)$ is even compact
and hence $B^*(A-\nu)^{-\frac12}$ is a compact operator
for $\nu<\min\sigma(A)$.
By Remarks~\ref{re:comp1} and \ref{re:comp2} one has $b_0=0$ and
$\sigmaess(\cM)=\sigmaess(D)=\essran u$.
\end{example}

\begin{example}
Let $\cH_1=\cH_2=L^2(0,1)$, let $q,u,v\in L^\infty(0,1)$, where $q$ and $u$
are real-valued, and consider the operators
\begin{alignat*}{2}
  Ay &= -y''+qy, \qquad & \dom(A) &= H^2(0,1)\cap H^1_0(0,1),
  \\[0.5ex]
  By &= (vy)', & \dom(B) &= H^1(0,1),
  \\[0.5ex]
  B^*y &= \ov{v}y', & \dom(B^*) &\supset H^1_0(0,1),
  \\[0.5ex]
  Dy &= uy, & \dom(D) &= L^2(0,1).
\end{alignat*}
Assumption~\ref{assump}.(I) is satisfied, and for $y\in\dom(\fra)=H^1_0(0,1)$
we have
\begin{align*}
  \|B^*y\|^2 &= \int_0^1 \bigl|\ov{v}y'\bigr|^2
  \le \sup|v|^2 \int_0^1 |y'|^2
  \\[0.5ex]
  &= \sup|v|^2\int_0^1\bigl(|y'|^2+q|y|^2\bigr) - \sup|v|^2\int_0^1 q|y|^2
  \\[0.5ex]
  &\le \sup|v|^2\int_0^1\bigl(|y'|^2+q|y|^2\bigr) - \sup|v|^2\cdot\inf q \int_0^1|y|^2
  \\[0.5ex]
  &= \sup|v|^2\fra[y] - \sup|v|^2\cdot\inf q\cdot \|y\|^2.
\end{align*}
Hence a possible choice for $a$ and $b$ is
\[
  a = -\sup|v|^2\cdot\inf q, \qquad b = \sup|v|^2.
\]
Clearly, $D$ is bounded and $\delta_-=\inf u$, $\delta_+=\sup u$.
Condition \eqref{condA1} is satisfied if and only if
\begin{equation}\label{ex2_condA1}
  v \not\equiv 0 \qquad\text{and}\qquad \sup u + \sup|v|^2 \le \inf q.
\end{equation}
If \eqref{ex2_condA1} holds, then $\delta_-\le\delta_+\le\inf q\le\alpha_-$ and hence
\[
  \mu_- = \inf u, \qquad
  \mu = \sup u+\sup|v|^2, \qquad
  \mu_+ \ge -\frac{a}{b} = \inf q.
\]
This, together with Theorem~\ref{th:specinB}, implies that
\[
  \sigma(\cM) \subset [\inf u,\sup u+\sup|v|^2] \cup [\inf q,\infty).
\]
It follows from \cite[Theorem~4.5]{ALMS94} (cf.\ Proposition~\ref{pr:sessALMS})
that
\[
  \sigmaess(\cM)=\essran(u+|v|^2).
\]
It is easy to see that the Schur complement is given by
\[
  S(z)y = -\biggl(\Bigl(1+\frac{|v|^2}{u-z}\Bigr)y'\biggr)'+qy-zy
\]
for $z$ with $\dist(z,\essran u)>b_0$.
Note that $\sigmaess(\cM)$ is the set of $z\in\CC$ for which
\[
  0\in\essran\Bigl(1+\frac{|v|^2}{u-z}\Bigr).
\]
In $(\mu,\infty)$ the spectrum of $\cM$ consists of a sequence of eigenvalues that
tends to $\infty$ and satisfies \eqref{est1} and \eqref{asympest1}, e.g.\
\[
  \lambda_n \ge \nu_{\kappa+n}-\sup|v|^2
  +\frac{\sup|v|^2\bigl(\inf q-\sup u-\sup|v|^2\bigr)}{\nu_{\kappa+n}-\sup u}
  +\cO\biggl(\frac{1}{\nu_{\kappa+n}^2}\biggr), \qquad n\to\infty,
\]
where $\nu_k$ are the eigenvalues of $A$.
If $\inf|v|^2>0$, then \eqref{Blowest} holds with
\[
  \hat a = -\inf|v|^2\cdot\sup q, \qquad \hat b = \inf|v|^2,
\]
and \eqref{est2} and \eqref{asympest2} are valid.
\end{example}

\medskip

\subsection*{Acknowledgements}
\rule{0ex}{1ex}\\
Both authors gratefully acknowledge the support of the
Engineering and Physical Sciences Research Council (EPSRC), grant no.\
EP/E037844/1.
M.~Strauss gratefully acknowledges the support from the Wales Institute of
Mathematical and Computational Sciences and the Leverhulme Trust grant: RPG-167.



\begin{thebibliography}{99}

\bibitem{AL95}{\scshape A.~Adamjan and H.~Langer},
Spectral properties of a class of rational operator valued functions,
\textit{J.\ Operator Theory} \textbf{33} (1995), 259--277.

\bibitem{AdamsFournier}{\scshape R.\,A.~Adams and J.\,F.~Fournier},
\textit{Sobolev Spaces}. Second edition.
Pure and Applied Mathematics (Amsterdam), vol.~140.
Elsevier/Academic Press, Amsterdam, 2003.

\bibitem{AMS09}{\scshape S.~Albeverio, A.\,K.~Motovilov and A.\,A.~Shkalikov},
Bounds on variation of spectral subspaces under $J$-self-adjoint perturbations.
\textit{Integral Equations Operator Theory} \textbf{64} (2009), 455--486.

\bibitem{AMT10}{\scshape S.~Albeverio, A.\,K.~Motovilov and C.~Tretter},
Bounds on the spectrum and reducing subspaces of a $J$-self-adjoint operator.
\textit{Indiana Univ.\ Math.\ J.} \textbf{59} (2010), 1737--1776.

\bibitem{ALMS94}{\scshape F.\,V.~Atkinson, H.~Langer, R.~Mennicken and A.\,A.~Shkalikov},
The essential spectrum of some matrix operators,
\textit{Math.\ Nachr.} \textbf{167} (1994), 5--20.

\bibitem{AJT05}{\scshape T.\,Ya.~Azizov, P.~Jonas and C.~Trunk},
Spectral points of type $\pi_+$ and $\pi_-$ of self-adjoint operators in Krein spaces.
\textit{J.\ Funct.\ Anal.} \textbf{226} (2005), 114--137.

\bibitem{BEL00}{\scshape P.~Binding, D.~Eschw\'e and H.~Langer},
Variational principles for real eigenvalues of self-adjoint operator pencils,
\textit{Integral Equations Operator Theory} \textbf{38} (2000), 190--206.

\bibitem{Bognar}{\scshape J.~Bogn\'ar},
\textit{Indefinite Inner Product Spaces}.
Ergebnisse der Mathematik und ihrer Grenzgebiete, vol.~78.
Springer-Verlag, New York--Heidelberg, 1974.

\bibitem{DES00}{\scshape J.~Dolbeault, M.~Esteban and E.~S\'er\'e},
On the eigenvalues of operators with gaps: Applications to Dirac operators,
\textit{J.\ Funct.\ Anal.} \textbf{174} (2000), 208--226.

\bibitem{EdmundsEvans}{\scshape D.\,E.~Edmunds and W.\,D.~Evans},
\textit{Spectral Theory and Differential Operators},
Clarendon Press, Oxford, 1987.

\bibitem{EL02}{\scshape D.~Eschw\'e and H.~Langer},
Triple variational principles for eigenvalues of self-adjoint operators and operator functions,
\textit{SIAM J.\ Math.\ Anal.} \textbf{34} (2002), 228--238.

\bibitem{EL04}{\scshape D.~Eschw\'e and M.~Langer},
Variational principles for eigenvalues of self-adjoint operator functions,
\textit{Integral Equations Operator Theory} \textbf{49} (2004), 287--321.

\bibitem{GLS99}{\scshape M.~Griesemer, R.\,T.~Lewis and H.~Siedentop},
A minimax principle for eigenvalues in spectral gaps:
Dirac operators with Coulomb potentials.
\textit{Doc.\ Math.} \textbf{4} (1999), 275--283.

\bibitem{JLT}{\scshape B.~Jacob, M.~Langer and C.~Trunk},
Variational principles for self-adjoint operator functions
arising from second order systems,
submitted, arXiv:1410.7083.

\bibitem{JT02}{\scshape P.~Jonas and C.~Trunk},
On a class of analytic operator functions and their linearizations.
\textit{Math.\ Nachr.} \textbf{243} (2002), 92--133.

\bibitem{katopert}{\scshape T.~Kato},
\textit{Perturbation Theory for Linear Operators},
Springer-Verlag, Berlin, 1995.

\bibitem{KLT04}{\scshape M.~Kraus, M.~Langer and C.~Tretter},
Variational principles and eigenvalue estimates for unbounded
block operator matrices and applications,
\textit{J.\ Comput.\ Appl.\ Math.} \textbf{171} (2004), 311--334.

\bibitem{L82}{\scshape H.~Langer},
Spectral functions of definitizable operators in Kre\u{\i}n spaces.
In: Functional Analysis (Dubrovnik, 1981), pp.~1--46,
\textit{Lecture Notes in Math.}, vol.~948, Springer, Berlin--New York, 1982.

\bibitem{LLMT05}{\scshape H.~Langer, M.~Langer, A.~Markus and C.~Tretter},
Spectrum of definite type of self-adjoint operators in Krein spaces,
\textit{Linear Multilinear Algebra} \textbf{53} (2005), 115--136.

\bibitem{LLMT08}{\scshape H.~Langer, M.~Langer, A.~Markus and C.~Tretter},
The Virozub--Matsaev condition and spectrum of definite type for
self-adjoint operator functions.
\textit{Complex Anal.\ Oper.\ Theory} \textbf{2} (2008), 99--134.

\bibitem{LLT02}{\scshape H.~Langer, M.~Langer and C.~Tretter},
Variational principles for eigenvalues of block operator matrices,
\textit{Indiana Univ.\ Math.\ J.} \textbf{51} (2002), 1427--1459.

\bibitem{LMM97}{\scshape H.~Langer, A.~Markus and V.~Matsaev},
Locally definite operators in indefinite inner product spaces.
\textit{Math.\ Ann.} \textbf{308} (1997), 405--424.

\bibitem{LMM05}{\scshape H.~Langer, A.~Markus and V.~Matsaev},
Self-adjoint analytic operator functions and their local spectral function,
\textit{J.\ Funct.\ Anal.} \textbf{235} (2005), 193--225.

\bibitem{LMMT01}{\scshape H.~Langer, A.~Markus, V.~Matsaev and C.~Tretter},
A new concept for block operator matrices: the quadratic numerical range,
\textit{Linear Algebra Appl.} \textbf{330} (2001), 89--112.

\bibitem{LMM90}{\scshape H.~Langer, R.~Mennicken, M.~M\"oller},
A second order differential operator depending non-linearly on the eigenvalue parameter.
in: \textit{Topics in Operator Theory: Ernst D. Hellinger Memorial Volume},
pp.~319--332, \textit{Oper.\ Theory Adv.\ Appl.}, vol.~48, Birkh\"auser, Basel, 1990.

\bibitem{LT98}{\scshape H.~Langer and C.~Tretter},
Spectral decomposition of some nonselfadjoint block operator matrices.
\textit{J.\ Operator Theory} \textbf{39} (1998), 339--359.

\bibitem{LS}{\scshape M.~Langer and M.~Strauss},
Triple variational principles for self-adjoint operator functions,
\textit{J.\ Funct.\ Anal.}, to appear.

\bibitem{MS96}{\scshape R.~Mennicken and A.\,A.~Shkalikov},
Spectral decomposition of symmetric operator matrices,
\textit{Math.\ Nachr.} \textbf{179} (1996), 259--273.

\bibitem{T09}{\scshape C.~Tretter},
Spectral inclusion for unbounded block operator matrices,
\textit{J.\ Funct.\ Anal.} \textbf{256} (2009), 3806--3829.

\bibitem{tretbook}{\scshape C. Tretter},
\textit{Spectral Theory of Block Operator Matrices and Applications},
Imperial College Press, London, 2008.

\bibitem{VM74}{\scshape A.~Virozub and V.~Matsaev},
The spectral properties of a certain class of selfadjoint operator-valued functions.
\textit{Functional Anal.\ Appl.} \textbf{8} (1974), 1--9.

\end{thebibliography}
\end{document}